\documentclass[12pt]{amsart}

\usepackage[letterpaper,left=1in,right=1in,top=0.8in,bottom=0.8in]{geometry}

\usepackage[utf8]{inputenc}
\usepackage[english]{babel}

\usepackage{amsmath}
\usepackage{amssymb}
\usepackage{amsthm}
\usepackage{color}
\usepackage[matrix,arrow,curve]{xy}

\usepackage{cite}

\usepackage{tikz}
\usetikzlibrary{matrix}
\usepackage{tikz-cd}

\usepackage{dsfont}

\usepackage{lipsum}

\usepackage[textsize=scriptsize]{todonotes}

\usepackage{hyperref}

\newcommand{\nc}{\newcommand}
\nc{\renc}{\renewcommand}
\nc{\ssec}{\subsection}
\nc{\sssec}{\subsubsection}

\newtheorem{theorem}[subsubsection]{Theorem}

\newtheorem*{theorem*}{Theorem}

\newtheorem{theor}{Theorem}

\newtheorem{lm}[subsubsection]{Lemma}

\newtheorem*{lm*}{Lemma}

\newtheorem{prop}[subsubsection]{Proposition}

\newtheorem{pro}[theor]{Proposition}

\newtheorem{cor}[subsubsection]{Corollary}

\newtheorem*{cor*}{Corollary} 

\newtheorem{co}[theor]{Corollary}

\theoremstyle{definition}
\newtheorem{df}[subsubsection]{Definition}

\newtheorem{rem}[subsubsection]{Remark}

\newtheorem{re}[theor]{Remark}

\newtheorem*{re*}{Remark}

\newtheorem{notation}[subsubsection]{Notation}

\newtheorem*{notation*}{Notation}

\newtheorem{step}{Step}

\makeatletter
\let\c@equation\c@subsubsection

\makeatother

\newcommand{\Maps}{\mathop{\mathrm{Maps}}\nolimits}
\newcommand{\Hom}{\mathop{\mathrm{Hom}}\nolimits}
\newcommand{\Loc}{\mathop{\mathrm{L}}\nolimits}
\newcommand{\Fin}{\mathop{\mathrm{Fin}}\nolimits}

\newcommand{\Gm}{\mathop{\mathbb{G}_m}\nolimits}
\newcommand{\spk}{\mathop{\mathrm{Spec} \, k}\nolimits}
\newcommand{\spl}{\mathop{\mathrm{Spec} \, L}\nolimits}

\newcommand{\Ker}{\mathop{\mathrm{Ker}}\nolimits}

\newcommand{\Coker}{\mathop{\mathrm{Coker}}\nolimits}
\newcommand{\coeq}{\mathop{\mathrm{coeq}}\nolimits}
\newcommand{\eff}{\mathop{\mathrm{eff}}\nolimits}

\newcommand{\SL}{\mathop{\mathrm{SL}}\nolimits}

\newcommand{\msl}{\mathop{\mathrm{MSL}}\nolimits}
\newcommand{\mgl}{\mathop{\mathrm{MGL}}\nolimits}
\newcommand{\hz}{\mathop{\mathrm{H}\widetilde{\mathbb{Z}}}\nolimits}
\newcommand{\gw}{\mathop{\mathrm{GW}}\nolimits}
\newcommand{\SH}{\mathop{\mathcal{SH}}\nolimits}
\newcommand{\SHh}{\mathop{\mathrm{SH}}\nolimits}

\newcommand{\Ee}{\mathop{\mathcal{E}}\nolimits}

\newcommand{\op}{\mathop{\mathrm{op}}\nolimits}
\newcommand{\gp}{\mathop{\mathrm{gp}}\nolimits}
\newcommand{\Gr}{\mathop{\mathrm{Gr}}\nolimits}
\newcommand{\GR}{\mathop{\widetilde{\mathrm{Gr}}}\nolimits}
\newcommand{\id}{\mathop{\mathrm{id}}\nolimits}
\newcommand{\red}{\mathop{\mathrm{red}}\nolimits}
\newcommand{\pr}{\mathop{\mathrm{pr}}\nolimits}
\newcommand{\const}{\mathop{\mathrm{const}}\nolimits}
\newcommand{\Th}{\mathop{\mathrm{Th}}\nolimits}

\newcommand{\Fr}{\mathop{\mathrm{Fr}}\nolimits}

\newcommand{\ZF}{\mathop{\mathbb Z \mathrm F}\nolimits}

\newcommand{\PSh}{\mathop{\mathrm{PSh}}\nolimits}
\newcommand{\Nis}{\mathop{\mathrm{Nis}}\nolimits}

\newcommand{\rs}{\mathop{\mathrm{RS}}\nolimits}

\newcommand{\Chw}{\mathop{\widetilde{\mathrm{CH}}}\nolimits}
\newcommand{\Cor}{\mathop{\widetilde{\mathrm{Cor}}}\nolimits}

\newcommand{\Z}{{\mathbb Z}}
\newcommand{\Ll}{{\mathcal L}}

\newcommand{\N}{{\mathbb N}}
\newcommand{\Pp}{{\mathbb P}}

\newcommand{\Gg}{{\mathbb G}}
\newcommand{\Aa}{{\mathbb A}}

\newcommand{\K}{{\mathrm{K}}}
\newcommand{\Kk}{{\mathbf{K}}}

\newcommand{\m}{{\mathfrak m}}
\newcommand{\Oo}{{\mathcal O}}

\newcommand{\E}{{\mathcal E}}
\newcommand{\C}{{\mathcal C}}

\newcommand{\Hrm}{{\mathrm H}}
\newcommand{\Tau}{{\mathcal T}}
\newcommand{\TAU}{\widetilde{{\mathcal T}}}
\newcommand{\chr}{\mathop{\mathrm{char}}\nolimits}
\newcommand{\supp}{\mathop{\mathrm{supp}}\nolimits}
\newcommand{\col}{\mathop{\mathrm{colim}}\nolimits}
\newcommand{\colim}{\mathop{\mathrm{colim}}}
\newcommand{\smk}{\mathop{\mathrm{Sm}_k}\nolimits}

\newcommand{\sus}{\mathop{\Sigma_{\mathbb P^1}^{\infty}}\nolimits}
\newcommand{\del}{\mathop{\Delta^{\bullet}_k}\nolimits}
\newcommand{\unit}{\mathop{\mathds{1}}\nolimits}

\title[The unit map of MSL]{The unit map of the \\ algebraic special linear cobordism spectrum}

\author{Maria Yakerson}
\address{Fakultät für Mathematik\\
Universität Regensburg\\
Universitätstr. 31\\
93040 Regensburg\\
Germany}
\email{\href{mailto:maria.yakerson@ur.de}{maria.yakerson@ur.de}}
\urladdr{\url{https://www.muramatik.com}}
\thanks{The author was supported by SFB/TR 45 “Periods, moduli spaces and arithmetic of algebraic varieties”}

\date{\today}

\begin{document}
\maketitle

\begin{abstract}
In joint work with Elmanto, Hoyois, Khan and Sosnilo~\cite{deloop3}, we computed infinite $\Pp^1$-loop spaces of motivic Thom spectra, using the technique of framed correspondences. This result allows us to express non-negative $\Gg_m$-homotopy groups of motivic Thom spectra in terms of geometric generators and relations. Using this explicit description, we show that the unit map of the algebraic special linear cobordism spectrum induces an isomorphism on $\Gg_m$-homotopy sheaves. 

\end{abstract}

\tableofcontents

\section{Introduction}

In algebraic topology, an important resource for analyzing the stable homotopy groups of spheres is given by the unit map of the complex cobordism spectrum $\mathrm{MU}$. This map has at least two features:
1) it induces an isomorphism on $\pi_0 = \Z$;
2) it detects nilpotence, giving rise to the field of chromatic homotopy theory. 

It would be interesting to know if similar techniques apply in motivic homotopy theory, for studying the motivic stable homotopy groups of spheres. There is not much yet known about motivic nilpotence phenomena (see the recent work of Bachmann and Hahn~\cite{BHNilpotence}). On the other hand, the abelian group $\pi_0$ of a spectrum is replaced by a richer invariant in motivic settings. For a motivic $\Pp^1$-spectrum $\E$ one considers a sequence of Nisnevich sheaves of abelian groups  $\{\underline{\pi}_0(\E)_l\}_{l \in \Z}$, called a homotopy module. One may ask an analogous question: does the unit map of a motivic cobordism spectrum induce an isomorphism of homotopy modules? 

The first guess would be to consider the unit map of the algebraic cobordism spectrum $\mgl$, which is the motivic analogue of $\mathrm{MU}$, constructed by Voevodsky~\cite{VoevodskyA1-HomThy}. As it turns out, the induced map on homotopy modules kills $\eta$, the motivic Hopf element. More precisely, Hoyois has shown that the unit map of $\mgl$ factors through the map $\unit_S / \eta \to \mgl$, which induces an isomorphism of homotopy modules~\cite[Theorem~3.8]{HoyoisAlgCobordism}. One could ask if there is another algebraic cobordism spectrum \glqq closer\grqq  \,to the motivic sphere spectrum $\unit_S$. Indeed, for the algebraic special linear cobordism spectrum $\msl$ (a motivic analogue of  $\mathrm{MSU}$, constructed by Panin and Walter~\cite{PW-MSLMSp}) the unit map induces an  isomorphism of homotopy modules. This can be shown by studying the geometry of oriented Grassmanians in a similar fashion to Hoyois' proof, as stated in~\cite[Example~16.34]{BHNorms}. However, we would like to understand this comparison in an explicit way. In this paper, we interpret both homotopy modules in terms of geometric generators and relations, and then compare them directly. 
 
Classically, the celebrated Pontryagin-Thom theorem identifies the $n$-th stable homotopy group of spheres with the group of $n$-dimensional smooth compact manifolds equipped with a trivialization of the stable normal bundle (so called framing), modulo the bordism equivalence relation. 
An approach for getting an analogous result for motivic stable homotopy groups was suggested by Voevodsky in his unpublished notes~\cite{VoevodskyNotesFrames}, where he introduced a notion of a 
a framed correspondence between smooth schemes $X$ and $Y$ over a base field $k$. In the simplest case $X = Y = \spk$ his construction gives a geometric version of framed points in topology. In more detail, a framed correspondence $c$ of level $n \geqslant 0$ is given by a closed subscheme $Z \subset \Aa^n_X$ (the support of $c$), finite over $X$; an \'etale neighborhood $U$ of $Z$ in $\Aa^n_X$; a morphism $\phi\colon U \to \Aa^n$, cutting out $Z$ as the preimage of $0$ (the framing of Z); and a morphism $g\colon U \to Y$. 
As Voevodsky observed, the set of framed correspondences $\Fr_n(X, Y)$ is in bijection with the set of morphisms of pointed Nisnevich sheaves from $(\Pp^1, \infty)^{\wedge n} \wedge X_+$ to $ \Loc_{\Nis} ((\Aa^1 / \Aa^1-0)^{\wedge n} \wedge Y_+)$. This bijection provides an explicit map
\[
\Theta_n \colon \Fr_n(X, Y) \to \Maps_{\SH(k)}(\Sigma^{\infty}_T X_+, \Sigma^{\infty}_T Y_+),
\]
where right-hand side is the mapping space  in the motivic stable homotopy $\infty$-category $\SH(k)$. 

In a series of papers, Ananyevskiy, Druzhinin, Garkusha, Neshitov and Panin  developed a theory of framed motives~\cite{GPFramedMotives}, \cite{AGPCancellation}, \cite{GPHomInvPresheaves}, \cite{GNPRelativeMotivicSpheres}, \cite{DPChar2}. As one of their main results, they computed infinite $\Pp^1$-loop spaces of  $\Pp^1$-suspension spectra in terms of framed correspondences, when $k$ is a perfect field. 
In particular, Garkusha and Panin proved in~\cite[Corollary~11.3]{GPFramedMotives} that for $l \geqslant 0$ the map $\Theta = \colim \Theta_n$ induces an isomorphism 
\[
\Coker(\ZF(\Aa^1_k, \Gg_m^{\wedge l}) \xrightarrow{i_1^*-i_0^*} \ZF(\spk, \Gg_m^{\wedge l})) \xrightarrow{\sim}  [\unit_k, \Sigma_{\Gg_m}^l \unit_k]_{\SHh(k)} = \pi_0(\unit_k)_l(k), 
\]
where $\ZF(X, Y)$ is the stabilized free abelian group on framed correspondences from $X$ to $Y$, modulo equivalences $c \sqcup d \sim c+d$. One can think of the left-hand side as of  $H_0(\ZF(\Delta^{\bullet}_k, \Gg_m^{\wedge l}))$, i.e. the zeroth homology of the framed version of the Suslin complex.
When the field $k$ has characteristic $0$, Neshitov has computed $H_0(\ZF(\Delta^{\bullet}_k, \Gg_m^{\wedge l}))$  as the Milnor-Witt K-theory $\K_l^{MW}(k)$~\cite{NeshitovFramedCorrMW}, recovering in that case the famous computation of the homotopy module of the motivic sphere spectrum by Morel~\cite[Theorem~6.40]{MorelA1-alg-top}. 

In  joint work with Elmanto, Hoyois, Khan and Sosnilo~\cite{deloop3}, we computed infinite $\Pp^1$-loop spaces of Thom spectra of virtual vector bundles of rank $0$ (more generally, of non-negative rank) in terms of generalizations of framed correspondences~\cite[Corollary~3.2.4]{deloop3}. As an application of this result, we can reinterpret the unit map of the spectrum $\msl$ on the level of $\Gg_m$-homotopy groups in terms of explicit geometric data. 

\begin{pro}[see Proposition~\ref{prop:unit map msl}]
\label{pro: unit map msl}
Let $k$ be a perfect field, and let $l \geqslant 0$.
Then the unit map $e_* \colon \pi_0(\unit_k)_l(k) \to \pi_0(\msl)_l(k)$ is canonically identified with the map 
\[
\varepsilon_* \colon H_0(\ZF(\Delta^{\bullet}_k, \Gg_m^{\wedge l})) \to H_0(\ZF^{\SL}(\Delta^{\bullet}_k, \Gg_m^{\wedge l})).
\]
\end{pro}

Here the right-hand side is constructed out of $\SL$-oriented framed correspondences, introduced in Section~\ref{ssec:SL-cor}. Such a correspondence of level $n$ is the same set of data as the usual framed correspondence, except that here a framing is a map $\phi \colon U \to \TAU_n$, where $\TAU_n \to \GR_n$ is the tautological bundle over the oriented Grassmanian $ \GR_n = \GR(n, \infty)$. The support is cut out as the preimage of the zero section of $\TAU_n$. There is a natural map $\varepsilon_n \colon \Fr_n(X, Y) \to \Fr_n^{\SL}(X, Y)$, given by embedding $\Aa^n \hookrightarrow \TAU_n$ as the fiber over the distinguished point of $\GR_n$. 
It induces a functor $\Fr_*(k) \to \Fr_*^{\SL}(k)$ between categories, where objects are smooth $k$-schemes and morphisms are given by ($\SL$-oriented) framed correspondences.

We prove the following comparison result, which was originally suggested by Ivan Panin.

\begin{theor}[see Theorem~\ref{thm:main}
]
\label{theo: compare ZF and ZF^SL}
Assume that $\chr k = 0$. Then the induced map 
\[
\varepsilon_* \colon H_0(\ZF(\Delta^{\bullet}_k, \Gg_m^{\wedge *})) \to H_0(\ZF^{\SL}(\Delta^{\bullet}_k, \Gg_m^{\wedge *}))
\]
 is an isomorphism of non-negatively graded rings.
\end{theor}
The surjectivity of $\varepsilon_*$ is proven by providing explicit $\Aa^1$-homotopies between framed correspondences, which allow us to deform an $\SL$-oriented framing so that its image is contained in the fiber over the distinguished point of $\GR_n$. 

To prove injectivity of $\varepsilon_*$, we employ the category $\Cor_k$ of finite Milnor-Witt correspondences of Calm\`es-Fasel~\cite{CFFiniteMWCor}.  This category has smooth $k$-schemes as objects, and a morphism from $X$ to $Y$ is, roughly speaking, given by a closed subscheme $Z \subset X \times Y$, finite and surjective over components of $X$, with an unramified quadratic form on $Z$. 

There is a functor $\alpha \colon \Fr_*(k) \to \Cor_k$, defined in~\cite[Proposition~2.1.12]{DF_MWMotComplexes}. We show that Neshitov's isomorphism $$H_0(\ZF(\Delta^{\bullet}_k, \Gg_m^{\wedge *})) \xrightarrow{\sim} \K_*^{MW}(k)$$ factors via $\alpha$ through the isomorphism $H_0(\Cor(\Delta^{\bullet}_k, \Gg_m^{\wedge *})) \xrightarrow{\sim} \K_*^{MW}(k)$, constructed in~\cite[Theorem~2.9]{CFComparisonMWCohom}. The functor $\alpha$ can be reinterpreted as follows: given a framed correspondence in $\Fr_n(X, Y)$, one considers the oriented Thom class of the trivial bundle of rank $n$ over $\spk$ (which is an element of $H_0^n(\Aa^n_k, \Kk_n^{MW})$), takes its pullback along the framing, and then applies the pushforward to $X \times Y$. Such functor is naturally extended to the category $\Fr^{\SL}_*(k)$, by applying the same procedure to the oriented Thom class of the tautological bundle over the oriented Grassmanian. Altogether, this allows us to define a left inverse map for $\varepsilon_*$.

From Theorem~\ref{theo: compare ZF and ZF^SL} we obtain the following straightforward corollaries.

\begin{co}[see Proposition~\ref{prop:homotopy module}]
\label{co: homo module}
Assume that $\chr k = 0$. Then the unit map 
$e \colon \unit_k \to \msl$ induces an isomorphism of the corresponding homotopy modules:
\begin{equation*}
e_* \colon \underline{\pi}_0(\unit_k)_* \xrightarrow{\sim} \underline{\pi}_0(\msl)_*.
\end{equation*}
\end{co}

The spectrum $\msl$ represents a cohomology theory with a special linear orientation, and as such has a universal  property~\cite[Theorem~5.9]{PW-MSLMSp}. In particular, a map of commutative monoids $\msl \to A$ in the homotopy category $\SHh(k)$ induces a special linear orientation of the cohomology theory $A^{*, *}$. Thus Corollary~\ref{co: homo module} immediately implies the following well-known fact.

\begin{co}[see Corollaries~\ref{cor:chow witt orientation} and~\ref{cor:mw-cohom}] \label{co: sl-orientation}
Assume that $\chr k = 0$. Then the Chow-Witt groups $H^*(-, \Kk^{MW}_*)$ and the Milnor-Witt motivic cohomology $H^{*, *}_{MW} (-, \Z)$ as ring cohomology theories acquire unique special linear orientations. 
\end{co}

\begin{re}
There is a recent work by Druzhinin and Kylling, currently in the status of a preprint, which extends the result of Neshitov (\textit{op. cit.}) to perfect fields $k$ of ${\chr k \ne 2}$~\cite[Sections 4, 5]{DKFiniteFields}. This result would imply that Theorem~\ref{theo: compare ZF and ZF^SL} also holds for such fields. Corollaries~\ref{co: homo module} and~\ref{co: sl-orientation} would hold over perfect fields of $\chr k>2$ as well, after inverting the characteristic of $k$. 
\end{re}

\subsection*{Notation}
Throughout the paper, $k$ is a perfect field. $\smk$ is the category of smooth separated schemes of finite type over $k$.
$\del$ is the standard cosimplicial object $n \mapsto \Delta^n_k$, where $\Delta^n_k = \spk[t_0, \dots, t_n]/(\sum_{i=0}^n t_i - 1)$ is the algebraic $n$-simplex. We write 
$\Aa^1 = \Aa^1_k$ and $\Pp^1 = \Pp^1_k$ when the field $k$ is fixed, and $\Gm = (\Aa^1-0, 1)$, $\Pp^1 = (\Pp^1, \infty) $ for pointed $k$-schemes. We denote by $z \colon X \hookrightarrow E$ the zero section of a vector bundle $E \to X$. We write $\Th_X(E) = E / E-z(X)$ for the Thom space of a vector bundle $E$ over a smooth scheme $X$. In particular, $T = \Aa^1 / \Aa^1-0$.

For an $\infty$-category $\C$, $\Maps_{\C}(x, y)$ denotes the space of morphisms from $x$ to $y$, and $[x, y]_{h\C} = \pi_0 \Maps_{\C}(x, y) $ denotes the set of morphisms in the homotopy category $h\C$. We denote by $\PSh(\C)$ the $\infty$-category of presheaves of spaces on $\C$.

We write $\Loc_{\Nis} \colon \PSh(\smk) \to \PSh_{\Nis}(\smk)$ for the left adjoint to the inclusion functor of Nisnevich sheaves, i.e. for the  Nisnevich sheafification. We write $\Loc_{\Aa^1} \colon \PSh(\smk) \to \PSh_{\Aa^1}(\smk)$ for the left adjoint to the inclusion functor of $\Aa^1$-invariant presheaves, i.e. for the so called (naive) $\Aa^1$-localization. It can be modelled as $(\Loc_{\Aa^1} P)(X) = \col_{n \in \Delta^{\op}} P(X \times \Delta^n_k).$

We denote by $\SH(k)$ the \textit{motivic stable homotopy $\infty$-category} of  $k$. $\SH(k)$ is constructed as the $\infty$-category of $\Pp^1$-spectra in pointed $\Aa^1$-invariant Nisnevich sheaves on $\smk$. $(\SH(k), \otimes)$ is a symmetric monoidal $\infty$-category under smash product, with the unit given by $\unit = \sus S^0_k$ (see~\cite[Section~4.1]{BHNorms}). We denote by $\SHh(k)$ the  homotopy category of $\SH(k)$. 

We denote by $\underline{\pi}_n(\E)_m$ the Nisnevich sheafification of the presheaf on $\smk$
\[\pi_n(\E)_m \colon U \mapsto 
[\Sigma_{S^1}^n \sus U_+, \Sigma_{\Gm}^m \E]_{\SHh(k)}\] for $\E \in \SH(k)$. Its value is naturally extended to essentially smooth $k$-schemes. 
We abbreviate $\pi_{n}(\E)_m(L) = \pi_n(\E)_m(\spl)$ for $\E \in \SH(k)$, $L/k$ a finitely generated field extension.

\subsection*{Acknowledgments}
The author would like to thank sincerely Alexey Ananyevskiy, Federico Binda, Elden Elmanto, Jean Fasel, Marc Hoyois, Adeel Khan, Lorenzo Mantovani, Alexander Neshitov, and Vova Sosnilo for helpful discussions. The author is very grateful to Tom Bachmann for careful reading of a draft of this paper. This work is part of author's PhD thesis under supervision of Marc Levine, and it could not be accomplished without his encouragement and inspiration, provided on daily basis.

\section{$E$-framed correspondences}

In this section we recall the defintion of an $E$-framed correspondence from~\cite[Section~2.2]{deloop3}\footnote{In~\textit{op. cit.} there was defined a stabilized version of $E$-framed correspondences, in a bigger generality, and they were called ``twisted equationally framed correspondences".}, which generalizes Voevodsky's original definition of a framed correspondence~\cite{VoevodskyNotesFrames}.
We recall functoriality properties of $E$-framed correspondences and related notions, generalizing the properties of framed correspondences studied in~\cite{GPFramedMotives}. Afterwards we recall from~\cite{deloop3} the computation of infinite $\Pp^1$-loop spaces of certain motivic Thom spectra via $E$-framed correspondences. 

\subsection{Main definitions and functoriality}

\begin{df}
\label{def: E-framed cor}
Let $X, Y$ be smooth $k$-schemes and $E$ a vector bundle over $Y$ of rank $r$. An \textit{$E$-framed correspondence}  $c=(U, \phi, g)$ of level $n \in \N$ from $X$ to $Y$ consists of the following data:
\begin{itemize}
\item a closed subscheme $Z \subset \Aa^{n+r}_X$, finite over $X$;
\item an \'etale neighborhood $p \colon U \to \Aa^{n+r}_X$ of $Z$;
\item a morphism $(\phi, g) \colon U \to \Aa^n \times E$ such that $Z$ as a closed subscheme of $U$ is the preimage of the zero section 
$z(0 \times Y) \subset \Aa^n \times E $.
\end{itemize}

We say that $E$-framed correspondences $(U, \phi, g)$ and  $(U', \phi', g')$  are equivalent if $Z = Z'$ and  $(\phi, g)$ coincides with $(\phi', g')$ in an \'etale neighborhood of $Z$ refining both $U$ and $U'$. 
We denote the set of $E$-framed correspondences modulo this equivalence relation as $\Fr_{E, n}(X, Y) $; in the case $E=Y$ we write $\Fr_{n}(X, Y)$. We call $Z$ the \textit{support} of $c$ and $\phi$ the \textit{framing} of $Z$.
\end{df}

\begin{rem}
When $E$ is a trivial bundle over $Y$ of rank $r$ this definition recovers the set of framed correspondences $\Fr_{n+r}(X, Y)$, introduced by Voevodsky in~\cite{VoevodskyNotesFrames} and later studied by Garkusha and Panin in~\cite{GPFramedMotives} (see also~\cite[Section~2.1]{deloop1}). 
\end{rem}

\sssec{}
\label{sssec:functor. of E-cor}
One can compose $E$-framed correspondences in the following way:
\begin{gather*}
\Fr_n(X, V) \times \Fr_{E, m}(V, Y) \longrightarrow \Fr_{E, n+m}(X, Y) \\
\big( (U, \phi, g), (W, \psi, h) \big) \mapsto (U \times_V W, \phi \times \psi, h \circ \pr_{W}). 
\end{gather*}

One can also compose with endomorphisms:
\begin{gather*}
\Fr_{E, n}(X, Y) \times \Fr_m(Y, Y)   \longrightarrow \Fr_{E, n+m}(X, Y) \\
\big( (W, \psi, h), (U, \phi, g) \big) \mapsto (W \times_Y U, \psi \times \phi, g \circ \pr_U), 
\end{gather*}
where $W \to Y$ is defined as $W \xrightarrow{h} E \to Y $.

\sssec{}
\label{sssec: external product}
The product of $E$-framed correspondences is defined as follows:
\begin{gather*}
\boxtimes \colon \Fr_{E,n}(X, Y) \times \Fr_{E',m}(X', Y') \longrightarrow \Fr_{E  \times E',n+m}(X \times X', Y \times Y') \\
 ((U, \phi, g), (U', \phi', g')) \mapsto  (U \times U', (\phi \circ \pr_U, \phi' \circ \pr_{U'}), g \times g').
\end{gather*}

\sssec{}
Recall from~\cite[Definition~2.3]{GPFramedMotives} the \textit{category of framed correspondences} $\Fr_*(k)$ which has smooth $k$-schemes as objects, and morphisms are given by 
$\Fr_*(X, Y) = \bigvee_{i=0}^{\infty} \Fr_i(X, Y),$
where each set $ \Fr_i(X, Y) $ is pointed by the correspondence with empty support $0_i \in  \Fr_i(X, Y)$. 
There is a canonical functor: 
$\gamma \colon \smk \to \Fr_*(k),$
which sends $f \colon X \to Y$ to the framed correspondence $(X, \const, f) \in \Fr_0(X, Y)$. By abuse of notation, we will consider morphisms of $k$-schemes as framed correspondences of level $0$.

For $X \in \smk$ consider the \textit{suspension morphism} $\sigma_X = (\Aa^1 \times X, \pr_{\Aa^1}, \pr_X) \in \Fr_1(X, X).$
The set of stabilized $E$-framed correspondences is given by 
\[\Fr_E(X, Y) = \col(\Fr_{E, 0}(X, Y) \xrightarrow{\sigma_Y} \Fr_{E, 1}(X, Y) \to \dots).\]

The original motivation for the definition of an $E$-framed correspondence comes from the following lemma, attributed to Voevodsky. 

\begin{lm}[Voevodsky]
\label{lm:voev thom}
Let $X$, $Y$ be smooth $k$-schemes, $E$ a vector bundle over $Y$ of rank $r$. Then there is a natural bijection:
$$\Theta_{E, n} \colon \Fr_{E, n}(X, Y) \xrightarrow{\sim} \Hom_{\PSh_{\Nis}(\smk)_{\bullet}} \big((\Pp^1, \infty)^{\wedge r+n} \wedge X_+, \Loc_{\Nis}(T^n \wedge \Th_Y(E))\big).$$
\end{lm}
\begin{proof}
This is a particular case of~\cite[Corollary~A.1.5]{deloop1}. The map $\Theta_{E, n}$ is constructed as follows. Let $c = (U, \phi, g) \in  \Fr_{E, n}(X, Y)$ have support $Z$, and $p \colon U \to \Aa^{n+r}_X$  be the  \'etale neighborhood of $Z$. Then the map $\Theta_{E, n}(c)$ is induced by the map of Nisnevich sheaves 
\[(\Pp^1)^{\times (n+r)} \times X = \Loc_{\Nis} \big(((\Pp^1)^{\times (n+r)} \times X - Z) \sqcup_{U-Z} U \big) \xrightarrow{\const \sqcup \overline{(\phi, g)}} \Loc_{\Nis} \big(\Aa^n \times E / (\Aa^n - 0) \times (E-Y) \big),\]
where $U \to (\Pp^1)^{\times (n+r)} \times X$ is defined by composing $p$ with the embeddings at the complement of infinity, and $\const$ is the constant map to the distinguished point.
\end{proof}

Under the bijection of Lemma~\ref{lm:voev thom}, the suspension morphism $\sigma_{\spk}$ corresponds to the canonical motivic equivalence of pointed Nisnevich sheaves $(\Pp^1, \infty) \xrightarrow{\sim} \Pp^1 / \Pp^1-0 \simeq T$. Hence we get an induced map
\begin{equation}
\label{eq:thetaE}
\Theta_E \colon \Fr_E(X, Y) \longrightarrow  \Maps_{\SH(k)} (\Sigma_{\Pp^1}^r \Sigma^{\infty}_{\Pp^1} X_+, \Sigma^{\infty}_T \Th_Y(E)),
\end{equation}
functorial in $X$.

\subsection{Infinite loop spaces of motivic Thom spectra}

~

\sssec{} \label{sssec:functor clopen}
$E$-framed correspondences have the following functoriality with respect to vector bundles.
Assume $f \colon E \to E' $ is a map of rank $r$ vector bundles over smooth $k$-schemes $Y$, $Y'$ respectively, which is injective on each fiber, i.e. the canonical morphism 
$z(Y) \to E \times_{E'} {z(Y')}$ is an isomorphism. Then $f$ induces the maps
$f_{*, n} \colon \Fr_{E, n}(-, Y) \longrightarrow \Fr_{E', n}(-, Y')$
and
$f_* \colon \Fr_{E}(-, Y) \to \Fr_{E'}(-, Y').$

We will need an extension of this functoriality to the category $\smk_+$, the full subcategory of smooth pointed $k$-schemes of the form $X_+$. Equivalently, $\smk_+$ is the category whose objects are smooth $k$-schemes and whose morphisms are partially defined maps with clopen domains. Let $f \colon E \dashrightarrow E'$ be a partially-defined map with a clopen domain, that is, $f \colon B \to E'$ where $E = B \sqcup B^c $. Assume that restriction to the zero section gives a map $ f\big|_{z(Y)} \colon A \to Y'$ where $z(Y) = A \sqcup A^c $, and that $A = B \times_{E'} z(Y') $. Then $f$ induces a map:
\begin{gather*}
f_{*, n}(X) \colon \Fr_{E, n}(X, Y) \longrightarrow \Fr_{E', n}(X, Y') \\
(U, \phi, g) \mapsto \big( g^{-1}(B), \phi\big|_{g^{-1}(B)}, f \circ g\big|_{g^{-1}(B)} \big),
\end{gather*}
functorial in $X \in \smk$, which gives $f_* \colon \Fr_{E}(-, Y) \to \Fr_{E'}(-, Y')$ after stabilization.

\sssec{}
This way we can define a structure of a $\Fin_*$-object on the presheaf $\Fr_E(-, Y)$. The category $\Fin_*$ of pointed finite sets is equivalent to the category with objects $\langle n \rangle = \{1, \dots, n \}$ for $n \geqslant 0$ and partially-defined maps. The functor \[F \colon \Fin_* \to \PSh(\smk); \quad
\langle n \rangle \mapsto \Fr_{E^{\sqcup n}}(-, Y^{\sqcup n})\]
is constructed as follows. Let $a \colon \langle n \rangle \dashrightarrow \langle m \rangle $ be a partially-defined map. The map $a$ induces a partially-defined map $\hat{a} \colon E^{\sqcup n} \dashrightarrow E^{\sqcup m} $ with a clopen domain, satisfying the requirements of the construction in Section~\ref{sssec:functor clopen}. We set  $F(a) = \hat{a}_*$.

The following form of additivity holds for $E$-framed correspondences.
\begin{prop}
\label{prop:E-additivity}
Let $Y_1, \dots, Y_m$ be smooth $k$-schemes, and let $E_1, \dots, E_m$ be vector bundles of rank $r$ over $Y_1, \dots, Y_m$ respectively. Then the canonical map 
\[\alpha \colon  \Fr_{E_1 \sqcup \dots \sqcup E_m}(-, Y_1 \sqcup \dots \sqcup Y_m) \to \Fr_{E_1}(-, Y_1) \times \dots \times \Fr_{E_m}(-, Y_m)\]
is an $\Aa^1$-equivalence, i.e. $\Loc_{\Aa^1}(\alpha)$ is an equivalence. In particular, for every  $Y \in \smk$ and a vector bundle $E$ over $Y$ the presheaf of spaces $\Loc_{\Aa^1} \Fr_E(-, Y)$ is an  $\Ee_{\infty}$-monoid in $\PSh(\smk)$.
\end{prop}

\begin{proof}
This is~\cite[Proposition~2.3.6]{deloop3}, the proof is the same as for the case $E_i$ trivial of rank $0$ for all $i$, which was proven in~\cite[Proposition~2.2.11]{deloop1}. The map $\beta$, inverse up to $\Aa^1$-homotopy to $\alpha$, is constructed as follows. Assume $m=2$. Define
\begin{gather*}
\beta_n(X) \colon \Fr_{E_1, n}(X, Y_1)  \times \Fr_{E_2, n}(X, Y_2) \longrightarrow \Fr_{E_1 \sqcup   E_2, n+1}(X, Y_1 \sqcup  Y_2) \\
(U, \phi, g) \times (W, \psi, h) \mapsto 
(\Aa^1_U \sqcup \Aa^1_W, \phi \times t_1 \sqcup \psi \times (t_2-1), g \circ \pr_U \sqcup \, h \circ \pr_V),
\end{gather*}
where $t_1$ and $t_2$ are the coordinate functions on each copy of $\Aa^1$. The proof in~\cite[Proposition~2.2.11]{deloop1} shows that $\col_i \Loc_{\Aa^1} \beta_{2i} $ is inverse to $\Loc_{\Aa^1} \alpha$.
\end{proof}

\sssec{}
One of the main inputs for our work is the following computation of infinite $\Pp^1$-loop spaces for motivic Thom spectra of stable vector bundles of rank $0$. For the $\infty$-categorical definition of group completion, see~\cite[Remark~4.5]{GGNUniversalityInfLoops}.

\begin{theorem}[Elmanto-Hoyois-Khan-Sosnilo-Yakerson]
\label{thm:map space thom}
Let $k$ be a perfect field, $Y$ a smooth $k$-scheme, $E$ a vector bundle over $Y$ of rank $r$. Then the map $\Theta_E$, constructed in~\eqref{eq:thetaE}, induces an equivalence of presheaves of spaces on $\smk$:
$$\Theta_E \colon \Loc_{\Nis} (\Loc_{\Aa^1} \Fr_E(-, Y))^{\gp} \xrightarrow{\sim}
\Maps_{\SH(k)} (\Sigma^{\infty}_{\Pp^1} (-)_+, 
\Sigma_T^{-r} \Sigma^{\infty}_T \Th_Y(E)),$$
where $\gp$ denotes group completion with respect to the  $\Ee_{\infty}$-structure from Proposition~\ref{prop:E-additivity}.
\end{theorem}

Note that Theorem~\ref{thm:map space thom} provides a fairly explicit model for infinite $\Pp^1$-loop spaces, because the Nisnevich sheafification of the presheaf $(\Loc_{\Aa^1} \Fr_E(-, Y))^{\gp}$ is an $\Aa^1$-invariant sheaf of grouplike $\Ee_{\infty}$-spaces, so there is no need to apply these localizations multiple times, as opposed to the general procedure of motivic localization. 

\begin{proof}
It is proven in~\cite[Corollary~3.2.4]{deloop3} that these presheaves of spaces are equivalent. By~\cite[Remark~3.2.5]{deloop3}, the equivalence is induced by the map $\Theta_E$: one reduces to the case of $E$ being a trivial bundle, and in that case the proof is given in~\cite[Corollary~3.3.8]{deloop2}. 

The proof of~\cite[Corollary~3.2.4]{deloop3} is based on structural properties of tangentially framed correspondences. A more straightforward proof of Theorem~\ref{thm:map space thom} can be found in~\cite[Theorem~2.2.2]{MuraThesis}. 
\end{proof}

\section{The unit map of $\msl$ via framed correspondences}

In this section, after recalling the construction of $\msl$, we introduce $\SL$-oriented framed correspondences, and interpret the unit map of $\msl$ in terms of a comparison map between framed correspondences and their $\SL$-oriented version. We then formulate the main result and explain its corollaries. 

\subsection{Recollection on $\msl$}

\sssec{} We briefly recall the construction of $\msl$ from~\cite[Section~4]{PW-MSLMSp}, to fix the notation. 
For $p \geqslant 1$ consider the Grassmanian $\Gr(n, np) = \Gr(n, (\Oo^{\oplus n}_k)^{\oplus p})$ and its tautological bundle $ \Tau(n, np)$. We denote the colimits along closed embeddings $\Gr_n = \col_{p} \Gr(n, np)$ and $\Tau_n = \col_{p} \Tau(n, np)$. The embedding $\Gr(n, n) \hookrightarrow \Gr(n, np)$ makes each $\Gr(n, np)$ a pointed scheme, and then $\Gr_n$ by taking colimit. 

For $n\geqslant 1$ consider the line bundle $\det(\Tau(n, np)) \to \Gr(n, np)$. The oriented Grassmanian is defined as \[\GR(n, np) = \det(\Tau(n, np)) - z(\Gr(n, np)) \in \smk.\]  
The projection $\pi_{n, np} \colon \GR(n, np) \to \Gr(n, np)$ is a principal $\Gm$-bundle. 
Define $\TAU(n, np) = \pi_{n, np}^*(\Tau(n, np)).$
Denote $\GR_n = \col_{p} \GR(n, np)$ and $\TAU_n = \col_{p} \TAU(n, np).$
By definition, \[\msl = \col_n \Sigma^{-n}_T \Sigma^{\infty}_T \Th_{\GR_n}(\TAU_n) \in \SH(k).\]

\sssec{}
The distinguished point $\Gr(n, n) \hookrightarrow \Gr(n, np)$ induces the map $$\Gg_m \simeq \Lambda^n \Oo_{\Gr(n, n)}^n - 0 \hookrightarrow \Lambda^n \Tau(n, np) - z(\Gr(n, np)) = \GR(n, np).$$
Each scheme $\GR(n, np)$ is pointed by $1 \in \Gg_m$, and so is the colimit $\GR_n$. 

There are canonical morphisms
$\widetilde{j}_{n, m} \colon \GR_n \times \GR_m \to \GR_{n+m} $, which induce isomorphisms
\begin{equation}
\label{eq:r_n,m}
\TAU_n \times \TAU_m \xrightarrow{\sim} \widetilde{j}_{n, m}^* \TAU_{n+m}.
\end{equation}

The inclusion $\GR(n, np) \subset \det(\Tau(n, np))$ gives a nowhere vanishing section of the line bundle  $\det\TAU(n, np)$, so defines a trivialization
\begin{equation}
\label{eq:sl-orientation}
\Oo_{\GR(n, np)} \xrightarrow{\sim} \det(\TAU(n, np)).
\end{equation}

\subsection{Zeroth homotopy group of $\msl$}
\sssec{} \label{sssec:deloop3 for msl}
For a smooth $k$-scheme $X$, define $\Fr_{\TAU_n, m}(X, \GR_n) =  \col_p  \Fr_{\TAU(n, np), m}(X, \GR(n, np)),$
and similarly for stabilized correspondences $\Fr_{\TAU_n}(X, \GR_n)$.
Since $\Sigma^{n}_T \Sigma^{\infty}_T X_+$ is a compact object in $\SH(k)$, the maps $\Theta_{\TAU(n, np)}(X)$
from Lemma~\ref{lm:voev thom} induce after taking colimit along $p$ the map
\[\Theta_{\TAU_n}(X) \colon \Fr_{\TAU_n}(X, \GR_n) \longrightarrow  \Maps_{\SH(k)}\big(\Sigma^{n}_T \Sigma^{\infty}_T X_+, \Sigma^{\infty}_T \Th_{\GR_n}(\TAU_n)\big).\]
We obtain from Theorem~\ref{thm:map space thom} the following corollary, using that $\Aa^1$-localization, Nisnevich sheafification and group completion are left adjoint functors, Nisnevich sheaves are closed under filtered colimits, and group completion commutes with Nisnevich sheafification by~\cite[Lemma~5.5]{HoyoisCdh}.
\begin{cor}
\label{cor:maps into msl}
The colimit of maps $\Theta_{\TAU_n}$ induces an equivalence of presheaves of spaces on $\smk$:
$$\Theta_{\TAU} \colon \Loc_{\Nis} (\Loc_{\Aa^1} \col_n \Fr_{\TAU_n}(-, \GR_n))^{\gp} \xrightarrow{\sim}  \Maps_{\SH(k)}\big( \Sigma^{\infty}_T (-)_+, \msl \big). $$
In particular, 
$$\pi_0(\msl)_0(k) \simeq
\pi_0 \big( \Loc_{\Aa^1} \col_n \Fr_{\TAU_n}(-, \GR_n)(\spk)\big) ^{\gp}, $$
where right-hand side is the classical group completion of a monoid. 
\end{cor}

\begin{df}
The abelian group of \textit{linear $E$-framed correspondences} from $X$ to $Y$ of level $n$ is defined as
\[\ZF_{E, n}(X, Y) = \Z \cdot \Fr_{E, n}(X, Y) / (c \sqcup d - c - d),\]
 where $c \sqcup d$ is given by the disjoint union of the data of correspondences $c$ and $d$, whose supports are disjoint as subschemes of $\Aa^{n+r}_X$. Note that $\ZF_{E, n}(X, Y)$ is isomorphic to  the free abelian group on $E$-framed correspondences with connected support.
 
  The pairing
  \[\ZF_{E, n}(X, Y) \times \ZF_m(Y, Y) \longrightarrow \ZF_{E, n+m}(X, Y),\] induced by the composition in Section~\ref{sssec:functor. of E-cor},
  allows one to define stabilization with respect to suspension:
$\ZF_E(X, Y) = \col(\ZF_{E, 0}(X, Y) \xrightarrow{\sigma_Y} \ZF_{E, 1}(X, Y) \to \dots).$
 \end{df}
 
 We now express the zeroth homotopy group of motivic Thom spectra of stable vector bundles of rank $0$ via linear $E$-framed correspondences.
 For $E$ a trivial vector bundle of rank $0$ this result is stated in~\cite[Corollary~11.3]{GPFramedMotives}:
 \begin{equation}
 \label{eq:sphere pi0 linear}
 \pi_0(\sus Y_+)_0(k) \simeq \Coker(\ZF(\Aa^1, Y) \xrightarrow{i_0^* - i_1^*} \ZF(\spk, Y))= H_0(\ZF(\Delta^{\bullet}_k \times X, Y)),
 \end{equation}
 where $\ZF(\Delta^{\bullet}_k \times X, Y)$ is a simplicial abelian group.

\begin{lm}
\label{lm:group completion linear}
Let $X, Y$ be smooth $k$-schemes, and $E$ a vector bundle over $Y$ of rank $r$. Then the following abelian groups are canonically isomorphic:
$$\pi_0(\Loc_{\Aa^1} \Fr_E(-, Y)(X))^{\gp} \simeq
H_0(\ZF_E(\Delta^{\bullet}_k \times X, Y)).$$
\end{lm}

\begin{proof}
 By definition, \[\pi_0(\Loc_{\Aa^1} \Fr_E(-, Y)(X)) \simeq
 \coeq\big( \Fr_E(\Aa^1_X, Y) \rightrightarrows \Fr_E(X, Y) \big).\]
 The monoid operation is induced by the following map:
\[
  \Fr_{E, n}(X, Y) \times  \Fr_{E, n}(X, Y) \xrightarrow{\beta_n(X)}  \Fr_{E \sqcup E, n+1}(X, Y \sqcup Y) \xrightarrow{(\id \sqcup \id)_*} \Fr_{E, n+1}(X, Y),
\]
where $\beta_n(X)$ was defined in the proof of Proposition~\ref{prop:E-additivity}. Since taking free abelian group on a set is a left adjoin functor, it preserves colimits.
Hence the group completion is computed as follows:
 $$\pi_0(\Loc_{\Aa^1} \Fr_E(-, Y)(X))^{\gp} \simeq
\Coker\big( \Z \cdot \Fr_E(\Aa^1 \times X, Y) \xrightarrow{i_0^* - i_1^*} \Z \cdot \Fr_E(X, Y) \big)  \big/ \thicksim_{s} \, ,$$ 
 where the equivalence relation $\thicksim_{s}$ is given by equivalences for each $c_1, c_2 \in \Fr_{E, n}(X, Y)$: 
 $$[(U_1, \phi_1, g_1)] +_s [(U_2, \phi_2, g_2)] \thicksim_{s}  
 [(U_1 \times \Aa^1 \sqcup U_2 \times \Aa^1, \phi_1 \times t_1 \sqcup \phi_2 \times (t_2-1), g_1 \sqcup g_2 \circ \pr_{U_1 \sqcup U_2})].$$
 Here $[-]$ denotes equivalence classes in the cokernel, and the right-hand side is the equivalence class of a correspondence in $\Fr_{E, n+1}(X, Y)$.

On the other hand, $\ZF_E(X, Y)$ is constructed as the quotient of the  free abelian group $\Z \cdot \Fr_E(X, Y)$, with equivalence relation given by the following equivalences for $c_1, c_2 \in \Fr_{E, n}(X, Y)$ with disjoint supports $Z_1$ and $Z_2$ in $\Aa^{n+r}_X$:
 $$(U_1, \phi_1, g_1) + (U_2, \phi_2, g_2) \sim 
 (U_1 \times \Aa^1 \sqcup U_2 \times \Aa^1, \phi_1 \times t_1 \sqcup \phi_2 \times t_2, g_1 \sqcup g_2 \circ \pr_{U_1 \sqcup U_2}).$$
Here the right-hand side belongs to $\Fr_{E, n+1}(X, Y)$, because we postcomposed the sum $c_1 + c_2$ with the suspension $\sigma_Y$.

As we can see, this equivalence relation is a priori different, but it is the same as $\thicksim_s$ up to $\Aa^1$-homotopy. Indeed, let $c_1, c_2 \in \Fr_{E, n}(X, Y)$ have supports $Z_1$ and $Z_2$ that are not disjoint. 
Then we can make them disjoint by suspending and applying an  $\Aa^1$-homotopy:
$$H = (U_2 \times \Aa^1 \times \Aa^1, \phi_2 \times (t  - s), g_2  \circ \pr_{U_2} ) \in \Fr_{E, n+1}(\Aa^1 \times X, Y), $$
where $s$ denotes the homotopy coordinate. 
This way we get:
$i_0^*(H) = \sigma_Y \circ c_2$, and $\supp(i_1^*(H)) = Z_2 \times 1$ is  disjoint with $Z_1 \times 0 = \supp(\sigma_Y \circ c_1)$ in $\Aa^{n+r+1}_X$.

Similarly, sums $+$ and $+_s$ are equivalent via the $\Aa^1$-homotopy in $\Fr_{E, n+1}(\Aa^1 \times X, Y)$:
$$H' = \big( (U_1 \times \Aa^1 \sqcup U_2 \times \Aa^1) \times \Aa^1, \phi_1 \times t_1 \sqcup \phi_2 \times (t_2-s), g_1 \sqcup g_2 \circ \pr_{U_1 \sqcup U_2} \big), $$
where $s$ denotes the homotopy coordinate. The claim follows.

\end{proof}

Combining Corollary~\ref{cor:maps into msl} and  Lemma~\ref{lm:group completion linear}, we get an explicit presentation of $ \pi_0(\msl)_0(k)$, since all the functors involved commute with filtered colimits. 
\begin{cor}
\label{cor:pi0 msl}
There is a canonical isomorphism of abelian groups:
\[ \pi_0(\msl)_0(k) \simeq \col_n H_0 \big( \ZF_{\TAU_n}(\del, \GR_n) \big) .\]
\end{cor}

\subsection{$\SL$-oriented framed correspondences} \label{ssec:SL-cor}
For future comparison with $\pi_0(\unit)_0(k)$, we now rewrite Corollary~\ref{cor:pi0 msl} in more convenient terms. 

\begin{df}
Let $X$, $Y$ be smooth $k$-schemes. The set of \textit{$\SL$-oriented framed correspondences} of level $n$ from $X$ to $Y$ is defined as  
$\Fr_n^{\SL}(X, Y) = \Fr_{\TAU_n \times Y, 0}(X, \GR_n \times Y)$. 
More concretely, an $\SL$-oriented framed correspondence $c=(U, \phi, g) \in \Fr_n^{\SL}(X, Y)$ is given by the following data:
\begin{itemize}
\item a closed subscheme $Z$ in $\Aa^n_X$, finite over $X$;
\item an \'etale neighborhood $p \colon U \to \Aa^n_X$ of $Z$;
\item a  morphism $\phi \colon U \to \TAU_n$
such that $Z$ as a closed subscheme of $U$ is the preimage of the zero section $z(\GR_n) \subset \TAU_n $;
\item a morphism $g \colon U \to Y$.
\end{itemize}
Here by a morphism $\phi \colon U \to \TAU_n$ we mean a map $U \to \col_p \TAU(n, np)$, represented by a morphism $\phi \colon U \to \TAU(n, np)$ for some $p$. 
\end{df}

\sssec{} 
\label{sssec: SL-composition}
As for framed correspondences, there is a composition law:
\begin{gather*}
\circ \colon \Fr_n^{\SL}(X, Y) \times \Fr_m^{\SL}(Y, V) \longrightarrow \Fr_{n+m}^{\SL}(X, V) \\
 ((U, \phi, g), (U', \phi', g')) \mapsto  (U \times_Y U',  s_{n, m} \circ (\phi \circ \pr_U, \phi' \circ \pr_{U'}), g' \circ \pr_{U'}),
\end{gather*}
where $s_{n, m} \colon \TAU_n \times \TAU_m \simeq j_{n, m}^* \TAU_{n+m}  \to \TAU_{n+m}$ is the composition of the isomorphism~\eqref{eq:r_n,m} and the projection. In the same way, the product of framed correspondences, defined in Section~\ref{sssec: external product}, generalizes to the product of $\SL$-oriented framed correspondences. Similarly, one can define the category of $\SL$-oriented framed correspondences $\Fr^{\SL}_*(k)$, which has smooth $k$-schemes as objects, and morphisms are given by
$\Fr_*^{\SL}(X, Y) = \bigvee_{i=0}^{\infty} \Fr_i^{\SL}(X, Y)$.

\sssec{} \label{sssec: functor E} Inclusion of the distinguished point into $\GR_n$ induces an embedding $\Aa^n \hookrightarrow \TAU_n$, which after restriction to the zero section gives $0 \hookrightarrow \GR_n$. For each $X, Y \in \smk$ this embedding induces a natural map between correspondences:
\begin{equation}
\label{eq:eps_n}
\Fr_n(X, Y) \hookrightarrow \Fr^{\SL}_n(X, Y),
\end{equation}
which respects the composition and induces a faithful functor
$\E \colon \Fr_*(k) \longrightarrow \Fr^{\SL}_*(k).$

\sssec{} The following generalization of Lemma~\ref{lm:voev thom} holds:
\begin{lm}
\label{lm:voev msl}
Let $X$, $Y$ be smooth $k$-schemes. Then there is a natural bijection:
\[\Theta_n^{\SL} \colon \Fr^{\SL}_n(X, Y) \xrightarrow{\sim} \Hom_{\PSh_{\Nis}(\smk)_{\bullet}} ((\Pp^1, \infty)^{\wedge n} \wedge X_+, \Loc_{\Nis}(\Th_{\GR_n}(\TAU_n) \wedge Y_+)).\]
\end{lm}
\begin{proof}
Morphisms into the Nisnevich sheafification of $\Th_{\GR(n, np)}(\TAU(n, np)) \wedge Y_+$ are computed as $\Fr_{\TAU(n, np) \times Y, 0}(X, \GR(n, np) \times Y)$ by~\cite[Corollary~A.1.5 and Remark~A.1.6]{deloop1}, and then one passes to the colimit along $p$.
\end{proof}

By stabilizing $\Theta_n^{\SL}$ with respect to suspension and using that $\Sigma^{\infty}_{\Pp^1} X_+$ is a compact object in $\SH(k)$, we get an induced map of presheaves on $\smk$:
\[\Theta^{\SL} \colon \Fr^{\SL}(-, Y)  \longrightarrow  \Maps_{\SH(k)} (\Sigma^{\infty}_{\Pp^1} (-)_+, \msl \otimes \sus Y_+).\]

\begin{df}
We define \textit{linear $\SL$-oriented framed correspondences} as 
 \[\ZF^{\SL}_n(X, Y) = \Z\cdot  \Fr^{\SL}_n(X, Y) /  (c \sqcup d - c - d).\]
\end{df}
The map \eqref{eq:eps_n} descends to the map
$\varepsilon_n \colon \ZF_n(X, Y) \hookrightarrow \ZF^{\SL}_n(X, Y).$
In particular, we can define an abelian group 
$\ZF^{\SL}(X, Y) = \col(\ZF_0^{\SL}(X, Y) \xrightarrow{\sigma_Y} \ZF^{\SL}_{1}(X, Y) \to \dots),$
and the induced homomorphism of abelian groups:
\begin{equation}
\label{eq:unit map varpesilon}
\varepsilon \colon  \ZF(X, Y) \to \ZF^{\SL}(X, Y).
\end{equation}

\subsection{The unit map via framed correspondences}

\begin{lm}
\label{lm:linear frames}
Let $V$ be a smooth $k$-scheme. Then the following presheaves of abelian groups on $\smk$ are canonically isomorphic:
$$\psi \colon \ZF^{\SL}(-, V) \xrightarrow{\sim} \col_n  \ZF_{\TAU_n \times V}(-, \GR_n \times V).$$
\end{lm}

\begin{proof}
To simplify notations, we assume that $V = \spk$, since the same argument applies for arbitrary $V \in \smk$. 
For a smooth $k$-scheme $X$ set
$$\psi_n(X) \colon \ZF^{\SL}_n(X, \spk) \to \ZF_{\TAU_n, 0}(X, \GR_n)$$
to be the identity map. 
Let $$\chi_n(X) \colon \ZF_{\TAU_n, r}(X, \GR_n) \to 
\ZF^{\SL}_{n+r}(X, \spk)$$
be the map induced by the embedding 
$ \Aa^r \times \TAU_n \hookrightarrow \TAU_{n+r}$ that restricts to the canonical embedding 
$0 \times \GR_n  \hookrightarrow \GR_{n+r}.$ 
Clearly, $\chi_n \circ \psi_n = \id $ (in this case $r=0$). For the other composition, consider  
$\alpha \in  \ZF_{\TAU_n, r}(X, \GR_n)$. Then we get
$$\sigma^r_{\GR_{n+r}} \big(\psi_{n+r} ( \chi_n (\alpha)) \big) = \delta^r( \alpha),$$
where $\delta$ denotes the suspension 
$\ZF_{\TAU_*}(X, \GR_*) \to \ZF_{\TAU_{*+1}}(X, \GR_{*+1})$. 
So, the correspondences $\alpha$ and $\psi_{n+r} (\chi_n (\alpha))$ become equivalent  after taking colimits with respect to $\sigma_{\GR_*}$ and $\delta$. Both maps $\psi_n(X)$ and $\chi_n(X)$ respect suspensions, and so stabilize to inverse maps $\psi(X)$ and $\chi(X)$, functorial in $X$.
\end{proof}

Lemma~\ref{lm:linear frames} allows us to rewrite Corollary~\ref{cor:pi0 msl} in the following way.
\begin{cor}
\label{cor:pi0 msl linear}
There is a canonical isomorphism of abelian groups:
\[ \pi_0(\msl)_0(k) \simeq H_0(\ZF^{\SL}(\Delta^{\bullet}_k, \spk)).\]
\end{cor}

\sssec{}
We can express in a similar form the group $\pi_0(\msl)_l(k) = [\unit, \Sigma_{\Gg_m}^l \msl]_{\SHh(k)}$ for $l \geqslant 0$. 
Let $V$ be a smooth $k$-scheme and let $\pr \colon V \times \GR(n, np) \to V$ be the projection to $V$. By applying the reasoning of Section~\ref{sssec:deloop3 for msl} to the vector bundles $\pr^* \TAU(n, np) \to V \times \GR(n, np)$, we obtain the following isomorphism of presheaves of spaces on $\smk$, generalizing Corollary~\ref{cor:maps into msl}:
\[
\col_n \Loc_{\Nis} (\Loc_{\Aa^1} \Fr_{V \times  \TAU_n}(-, V \times \GR_n))^{\gp} \xrightarrow{\sim}  \Maps_{\SH(k)}\big( \Sigma^{\infty}_T (-)_+,  \Sigma^{\infty}_T  V_+ \otimes \msl \big).
\]

Applying Lemma~\ref{lm:group completion linear} to $X = \spk$, $Y_p = V \times \GR(n, np)$,  $E_p =V \times \TAU(n, np)$, and taking colimit with respect to $p$ expresses the abelian group $[\unit, \Sigma^{\infty}_T  V_+ \otimes \msl]_{\SHh(k)}$ as
$$\col_n  \Coker\big( \ZF_{V \times \TAU_n}(\Aa^1,V \times \GR_n) \xrightarrow{i_1^* - i_0^*} \ZF_{V \times \TAU_n}(\spk, V \times \GR_n) \big).$$

By Lemma~\ref{lm:linear frames} we get:
$$ [\unit, \Sigma^{\infty}_T  V_+ \otimes \msl]_{\SHh(k)} \simeq 
H_0(\ZF^{\SL}(\Delta^{\bullet}_k, V)).$$

In particular, for $l \geqslant 0$ we get
\begin{equation} 
\label{eq:Gm^times l}
[\unit, \Sigma^{\infty}_T  (\Gg_m^l)_+ \otimes \msl]_{\SHh(k)} \simeq 
H_0(\ZF^{\SL}(\Delta^{\bullet}_k, \Gg_m^l)).
\end{equation} 
Moreover, we deduce 
\begin{equation} 
\label{eq:Gm^l}
 [\unit,  \Sigma_{\Gg_m}^l \msl]_{\SHh(k)} \simeq 
 H_0(\ZF^{\SL}(\Delta^{\bullet}_k, \Gg_m^{\wedge l})),
 \end{equation}
where the right-hand side denotes the zeroth homology of the simplicial abelian group 
\[\ZF^{\SL}(\Delta^{\bullet}_k, \Gg_m^{\wedge l}) = \Coker \big( \bigoplus_{i=1}^l \ZF^{\SL}(\Delta^{\bullet}_k, \Gg_m^{l-1}) \xrightarrow{\oplus_{i=1}^l (j_i)_*} \ZF^{\SL}(\Delta^{\bullet}_k, \Gg_m^l) \big),\]
with the maps induced by embeddings $j_i \colon \Gg_m^{l-1} \hookrightarrow \Gg_m^l$, inserting $1$ at $i$-th place. 
Indeed,~\eqref{eq:Gm^l} follows from~\eqref{eq:Gm^times l} because 
the cofiber sequence
\[\bigoplus_{i=1}^l  \Sigma^{\infty}_T  (\Gg_m^{l-1})_+ \xrightarrow{\oplus_{i=1}^l (j_{i})_{ *}}
 \Sigma^{\infty}_T  (\Gg_m^l)_+ \longrightarrow
 \Sigma^{\infty}_T  \Gg_m^{\wedge l}\]
 splits in $\SHh(k)$, and $\ZF^{\SL}(\Delta^{\bullet}_k, \Gg_m^{\wedge l})$ is a direct summand of the simplicial group $\ZF^{\SL}(\Delta^{\bullet}_k, \Gg_m^l)$. 
 
 \sssec{}
 We now compare this expression with the formula \eqref{eq:sphere pi0 linear} for $\pi_0(\unit)_0(k)$. Recall that the unit map $e \colon \unit \to \msl$ is induced by the embeddings of distinguished points in $\GR_n $, giving $e_n \colon T^n \hookrightarrow \Th_{\GR_n}(\TAU_n)$. For a smooth $k$-scheme $V$ we have a commutative diagram of presheaves of spaces on $\smk$:
$$\xymatrix{
\Loc_{\Nis} (\Loc_{\Aa^1}  \Fr(-, V))^{\gp}  \ar[d]_{(\varepsilon_n)_*} \ar[r]^{\Theta_V}_{\sim} & \Maps_{\SH(k)}\big( \Sigma^{\infty}_T (-)_+, \Sigma^{\infty}_T V_+ ) \ar[d]^{\id \otimes (e_n)_*}  \\
\Loc_{\Nis} (\Loc_{\Aa^1}  \Fr_{V \times \TAU_n}(-, V \times \GR_n))^{\gp} \ar[r]^-{\Theta_{V \times \TAU_n}}_-{\sim}  & \Maps_{\SH(k)}\big( \Sigma^{\infty}_T (-)_+, \Sigma^{\infty}_T V_+ \otimes \Sigma^{-n}_T \Th_{\GR_n}(\TAU_n)).\\
}
$$
Here the left vertical morphism is induced by the stabilization of the maps: 
$$\varepsilon_{n, r} \colon \Fr_{n+r}(-, V) \to \Fr_{V \times \TAU_n, r} (-, V \times \GR_n), $$
given by embeddings $\Aa^n \hookrightarrow \TAU_n $ over the  distinguished point of $\GR_n$.

\sssec{} After taking colimit and applying~\eqref{eq:Gm^l}, we get the following geometric interpretation of the unit map on $\Gg_m$-homotopy groups of $\msl$.

\begin{prop}
\label{prop:unit map msl} 
The unit map $e \colon \unit \to \msl  $ induces the following commutative diagram of abelian groups for $l \geqslant 0$:
$$\xymatrix{
H_0(\ZF(\Delta^{\bullet}_k,  \Gg_m^{\wedge l}))  \ar[d]_{\varepsilon_*} \ar[r]^-{\Theta_*}_-{\sim} & \pi_0(\unit)_l(k) \ar[d]^{e_*}  \\
H_0(\ZF^{\SL}(\Delta^{\bullet}_k,  \Gg_m^{\wedge l})) 
\ar[r]^-{\,(\Theta^{\SL})_* \,}_-{\sim}  & \pi_0(\msl)_l(k)\\
}
$$
Here horizontal maps are induced by corresponding versions of Voevodsky's lemma (Lemma~\ref{lm:voev thom}  and Lemma~\ref{lm:voev msl}), and the left vertical map is induced by the homomorphism 
$\varepsilon \colon \ZF(\Delta^{\bullet}_k,  \Gg_m^{\wedge l}) \to \ZF^{\SL}(\Delta^{\bullet}_k,  \Gg_m^{\wedge l})$, defined in \eqref{eq:unit map varpesilon}.
\end{prop}

\subsection{Framed correspondences and Milnor-Witt K-theory}
\label{ssec: neshitov}

To study the graded abelian group $H_0(\ZF(\Delta^{\bullet}_k,  \Gg_m^{\wedge *}))$, one first defines a ring structure. 
 As shown in~\cite[Section~3]{NeshitovFramedCorrMW}, the product of framed correspondences, defined in~\eqref{sssec: external product},
 descends to a product 
$$H_0(\ZF(\del \times X, Y)) \times H_0(\ZF(\del \times X', Y')) \to H_0(\ZF(\del \times X \times X', Y \times Y'))$$
for any $X$, $Y$, $X'$, $Y' \in \smk $.
Taking $X = X' = \spk$, $ Y = \Gg_m^n, Y' =  \Gg_m^m$, we get a multiplicative structure on the graded abelian group $H_0(\ZF(\del, \Gg_m^*))$, which descends to a multiplication on 
$H_0(\ZF(\del, \Gg_m^{\wedge *}))$. The main result of~\cite{NeshitovFramedCorrMW} is the following theorem.
\begin{theorem}[Neshitov]
Let $k$ be a field of characteristic $0$. Then the following graded rings are isomorphic:
$$H_0(\ZF(\del, \Gg_m^{\wedge *})) \simeq \K_{\geqslant 0}^{MW}(k),$$
where $\K_{\geqslant 0}^{MW}(k)$ denotes the non-negative part of the Milnor-Witt K-theory of the field $k$. 
\end{theorem}
\sssec{} \label{sssec: varepsilon_*}
By the same argument as in~\cite[Section~3]{NeshitovFramedCorrMW}, the product of $\SL$-oriented framed correspondences induces a multiplication on $H_0(\ZF^{\SL}(\del, \Gg_m^{\wedge *}))$. The homomorphism 
$$\varepsilon_* \colon H_0(\ZF(\del, \Gg_m^{\wedge *})) \longrightarrow 
H_0(\ZF^{\SL}(\del, \Gg_m^{\wedge *})) $$
is then a graded ring homomorphism. We refer to  $\varepsilon_*$  as \textit{unit map}, motivated by Proposition~\ref{prop:unit map msl}.

\subsection{Main theorem and applications}

Our main result is the computation of the unit map $\varepsilon_*$.

\begin{theorem}
\label{thm:main}
Let $k$ be a field of characteristic $0$.  Then the unit map $\varepsilon_*$ is a graded ring isomorphism:
\[\varepsilon_* \colon H_0(\ZF(\del, \Gg_m^{\wedge *})) \xrightarrow{\sim} H_0(\ZF^{\SL}(\del, \Gg_m^{\wedge *})).\]
\end{theorem}

We will prove Theorem~\ref{thm:main} in the next section. Meanwhile we deduce immediate applications. Being corollaries of Theorem~\ref{thm:main}, our proofs work over fields of characteristic $0$, however Proposition~\ref{prop:homotopy module} holds over an  arbitrary base scheme~\cite[Example~16.34]{BHNorms}, hence so do Corollaries~\ref{cor:chow witt orientation} and~\ref{cor:mw-cohom}. 
 
\sssec{} 
Recall that Voevodsky defined the homotopy $t$-structure on $\SHh(k)$, whose heart $\SHh^{\heartsuit}(k)$ is equivalent to the category of homotopy modules $\Pi_*(k)$ (see~\cite[Section~5.2]{MorelIntroTrieste}). A \textit{homotopy module} is a sequence of strictly  $\Aa^1$-invariant Nisnevich sheaves of abelian groups $\{\E_i\}_{i \in \Z}$ with isomorphisms $\E_i \xrightarrow{\sim} (\E_{i+1})_{-1}$, and a morphism of homotopy modules is a sequence  of maps of sheaves, compatible with the isomorphisms. Here $\E_{-1}$ denotes the \textit{contraction} of $\E$:
$ \E_{-1}(X) = \Coker(\E(X) \xrightarrow{i_*} \E(X \times \Gg_m)), $
where $i$ is the embedding at $1 \in \Gg_m$.
The functor
$\SHh(k) \longrightarrow  \Pi_*(k)$
that sends $E$ to $\underline{\pi}_0(E)_*$
induces an equivalence after restriction to $\SHh^{\heartsuit}(k)$. Its quasi-inverse functor is denoted by $\Hrm \colon \Pi_*(k) \rightarrow  \SHh^{\heartsuit}(k)$.

\begin{prop}
\label{prop:homotopy module}
Let $k$ be a field of characteristic $0$. Then the unit map 
$e \colon \unit \to \msl$ induces an isomorphism of homotopy modules:
\begin{equation*}
e_* \colon \underline{\pi}_0(\unit)_* \xrightarrow{\sim} \underline{\pi}_0(\msl)_*.
\end{equation*}
\end{prop}

\begin{proof}
As follows from Proposition~\ref{prop:unit map msl} together with Theorem~\ref{thm:main}, for any finitely generated field extension $L/k$ and $l \geqslant 0$ the unit map induces an isomorphism
$$e_l(L) \colon \underline{\pi}_0(\unit_k)_l(L) \simeq [\mathds{1}_L, \Sigma_{\Gg_m}^l \mathds{1}_L]_{\SHh(L)} 
\xrightarrow{\sim}  [\mathds{1}_L, \Sigma_{\Gg_m}^l \msl_L]_{\SHh(L)} 
\simeq \underline{\pi}_0(\msl_k)_l(L).$$
The first and last isomorphisms follow from the fact that (suspended) spectra $\unit$ and $\msl$ are \textit{absolute} in the sense of~\cite[Definition~1.2.1]{DegliseOrientation}. Since $p \colon \spl \to \spk$ is an essentially smooth $k$-scheme, one can express it as a cofiltered limit of smooth $k$-schemes $p_{\alpha} \colon X_{\alpha} \to \spk$.
Then for any absolute spectrum $E$ one has:
\begin{multline*}
[\mathds{1}_L, E_L]_{\SHh(L)} \simeq  [\mathds{1}_L, p^*(E_k)]_{\SHh(L)} \simeq \colim_{\alpha} [\mathds{1}_{X_{\alpha}}, p_{\alpha}^*(E_k)]_{\SHh(X_{\alpha})} \simeq \\
\colim_{\alpha} [p_{\alpha, \sharp}(\mathds{1}_{X_{\alpha}}), E_k]_{\SHh(k)} = \colim_{\alpha} \pi_0(E_k)_0(X_{\alpha}) \simeq
\pi_0(E_k)_0(L) = \underline{\pi}_0(E_k)_0(L).
\end{multline*}
Here the second isomorphism is the content of~\cite[Lemma~A.7(1)]{HoyoisAlgCobordism}, and the rest follows from definitions.

Since $\SHh^{\heartsuit}(k)$ is an abelian category, the maps $e_l$ of strictly $\Aa^1$-invariant Nisnevich sheaves have kernels and cokernels which are also strictly $\Aa^1$-invariant sheaves, hence  unramified~\cite[Example~2.3]{MorelA1-alg-top}. In case $l \geqslant 0$, we have shown that $\Ker e_l(L) = \Coker e_l (L) = 0$ for all finitely generated field extensions $L/k$, which implies that $\Ker e_l$ and $\Coker e_l$ are zero sheaves. Hence $e_*$ is an isomorphism of the sheaves $\underline{\pi}_0(-)_l$, for $l \geqslant 0$. 

Finally, by definition of a morphism of homotopy modules, $e_*$ is compatible with contraction isomorphisms, so the fact that $e_l$ are isomorphisms for all $l \geqslant 0$ implies that $e_*$ is an isomorphism on each level $l \in \Z$.
\end{proof}

\sssec{} 
Recall that a \textit{special linear orientation} of a bigraded ring cohomology theory on the category $\smk$ is an extra structure that encodes the data of natural multiplicative Thom isomorphisms for vector bundles with trivialized determinants over smooth schemes (see~\cite[Definition~5.1]{PW-MSLMSp}). A homomorphism of commutative monoids $\msl \to A$ in $\SHh(k)$ induces a special linear orientation of the cohomology theory $A^{*,*}$~\cite[Theorem~5.5]{PW-MSLMSp}. We call such homomorphism an $\SL$-orientation of the ring spectrum $A$.

\begin{cor}
\label{cor:chow witt orientation}
The bigraded ring cohomology theory $H^*(-, \Kk^{MW}_*)$ carries a unique special linear orientation. In this sense, Chow-Witt groups are uniquely specially linearly oriented.
\end{cor}

\begin{proof}
The cohomology theory $H^*(-, \Kk^{MW}_*)$ is represented by the spectrum ${\Hrm \underline{\pi}_0(\unit)_* \in \SHh(k)^{\heartsuit}}$ in $\SHh(k)$. The sequence of ring homomorphisms
$$\msl \xrightarrow{\underline{\pi}_0} \Hrm \underline{\pi}_0(\msl)_* \xrightarrow{e_*^{-1}} \Hrm \underline{\pi}_0(\unit)_*$$
provides an $\SL$-orientation of the spectrum $\Hrm \underline{\pi}_0(\unit)_*$, and hence of the bigraded cohomology theory $H^*(-, \Kk^{MW}_*)$. Since any map of commutative monoids $\msl \to \Hrm \underline{\pi}_0(\unit)_*$ factors through $\Hrm \underline{\pi}_0(\msl)_*$ and is compatible with the unit map of $\msl$, this $\SL$-orientation is unique.
\end{proof}

\begin{rem}
The system of compatible Thom isomorphisms for $H^*(-, \Kk^{MW}_*)$ was constructed in~\cite[Theorem~4.2.7]{AHpi_0}.
\end{rem}

\begin{cor}
\label{cor:mw-cohom}
The spectrum $\hz$, representing MW-motivic cohomology, is uniquely $\SL$-oriented.
\end{cor}

\begin{proof}
By~\cite[Theorem~5.2]{BFEffectivity}, $\hz \simeq \tau^{\eff}_{\leqslant 0}(\unit)$, where the right-hand side denotes the image of $\unit$ in $\SHh(k)^{\eff}_{\leqslant 0}$.
Since $\msl$ is an effective spectrum, the unit map of $\msl$ induces a morphism
$$e_* \colon \tau^{\eff}_{\leqslant 0}(\unit) \longrightarrow \tau^{\eff}_{\leqslant 0}(\msl),$$
which, as we claim, is an equivalence in $\SHh(k)$.
Indeed, by~\cite[Proposition~4.(1)]{BachmannSlices} it is enough to check that $e_*$ induces an isomorphism of $\underline{\pi}_*(-)_0$. 
But both $\unit$ and $\msl$ belong to $\SHh(k)_{\geqslant 0}$, so the only of these sheaves of homotopy groups that survive after applying the functor $\tau^{\eff}_{\leqslant 0}$ are $\underline{\pi}_0(-)_0$. By Proposition~\ref{prop:homotopy module}, $e_*$ induces an isomorphism of $\underline{\pi}_0(-)_0$.

Hence the unique $\SL$-orientation of $\hz$ is given by the following sequence of ring homomorphisms:
$$\msl \xrightarrow{\tau^{\eff}_{\leqslant 0}} \tau^{\eff}_{\leqslant 0}(\msl) \xrightarrow{e_*^{-1}} \tau^{\eff}_{\leqslant 0}(\unit) \simeq \hz.$$
\end{proof}

\section{Computation of the unit map}

In this section, we prove that the unit map $\varepsilon_* \colon H_0(\ZF(\del, \Gg_m^{\wedge *})) \to H_0(\ZF^{\SL}(\del, \Gg_m^{\wedge *}))$ is an isomorphism in characteristic $0$. To prove surjectivity, we construct explicit $\Aa^1$-homotopies between framed correspondences. To proving injectivity, we employ the computation of $H_0(\ZF(\del, \Gg_m^{\wedge *}))$ by Neshitov~\cite{NeshitovFramedCorrMW} and the theory of Milnor-Witt correspondences of Calm\`es and Fasel~\cite{CFFiniteMWCor}.

\subsection{Surjectivity of the unit map $\varepsilon_*$}
\label{ssec: surjectivity}
\begin{notation}
In this subsection we will use the following abbreviations.
\begin{itemize}
\item $L / k$ is a finite field extension.
\item $s \in X(L)$ and the corresponding $L$-rational point of $X_L = X \times \spl$ are denoted the same way, for $X \in \smk$.
\item $c \sim c'$ denotes equality of classes of $\SL$-oriented linear framed correspondences $c$ and $c'$ in $H_0(\ZF^{\SL}(\del, Y))$.
\end{itemize}
\end{notation}

\sssec{} Fix a smooth $k$-scheme $Y$. We will show that for any $c \in \ZF^{\SL}_n(\spk, Y)$ there is $c' \in \ZF_n(\spk, Y)$ such that $c \sim \varepsilon(c')$ in $H_0(\ZF^{\SL}(\del, Y))$. This result for $Y = \Gg_m^l$ for all $l \geqslant 0$ implies the surjectivity of the unit map $\varepsilon_*$.

We can assume that $c$ is represented by an $\SL$-oriented framed correspondence with a connected support. That is, 
$\supp(c)_{\red} = \spl$, where $L$ is some finite extension of $k$. We can also assume that $c$ is of level $n>0$. We will use the following preliminary lemmas, analogous to~\cite[Section~2]{NeshitovFramedCorrMW}.

\begin{lm}
\label{lm:projection}
Let $c = (U, \phi, g)$ be a correspondence in $Fr_n^{\SL}(\spk, Y)$ with support $Z$ such that $Z_{\red} = \spl$. Then one can refine $U$ to $U'$, an \'etale neighborhood of $Z$ such that there is a projection $U' \to\mathrm{Spec} \, L$.
\end{lm}

\begin{proof}
It is enough to show that there is a projection from the henselization  $(\Aa^n_k)^h_Z$, so we can assume $Z = \spl$. 
Since $L/k$ is a separable field extension, the projection $\Aa^n_L \to \Aa^n_k$ is an \'etale neighborhood of $Z$, so we can consider the composition of projections: $(\Aa^n_k)^h_Z \to \Aa^n_L \to \spl$.
\end{proof}

\begin{lm}
\label{lm:sl}
Let $c = (U, \phi, g)$ be a correspondence in $Fr_n^{\SL}(\spk, Y)$. Assume there is a map $h \colon U \to \spl$. Let $A \in \SL(L) = \col_i \SL_i(L)$ and assume there is given an action of $\SL(L)$ on $\TAU_{n, L}$, that induces an endomorphism of the zero section. Denote by $A \cdot \phi$ the composition $U \xrightarrow{\phi \times h} \TAU_{n, L} \xrightarrow{A} \TAU_{n, L} \xrightarrow{pr} \TAU_n$.
Then $c \sim c' = (U, A \cdot \phi, g)$.
\end{lm}

\begin{proof}
The group $\SL(L)$ is generated by elementary matrices,
hence there is an homotopy $H(t) \colon \Aa^1 \to \SL$ such that $H(1) = A$ and $H(0) = E$ is the identity matrix. 
The data $d = (U \times \Aa^1, H(t) \cdot \phi, g \circ \pr_U)$ define a correspondence in $Fr_n^{\SL}(\Aa^1, Y)$, because its support $Z \times \Aa^1$ is finite over $\Aa^1$.
Since $i_0^*(d) = c$ and $i_1^*(d) = (U, A \cdot \phi, g)$, the lemma follows.
\end{proof}

\begin{prop}
\label{prop:epi}
For $n > 0$ let $c = (U, \phi, g) \in \Fr_n^{\SL}(\spk, Y)$ be a correspondence with support $Z$ such that $Z_{\red} = \spl$.
 Then there is $c' \in \Fr_n(\spk, Y)$ such that $c \sim \varepsilon(c')$.
\end{prop}

\begin{proof}
We consider $\Gr_n$ and $\GR_n$ as embedded in $\Tau_n$ and $\TAU_n$ via the  respective zero sections.
Denote $p \in \Gr_n$ the distinguished point.
Then the distinguished point $q \in \GR_n$ is $1 \in \Aa^1-0$ in the fiber of 
$\pi_n \colon \GR_n \to \Gr_n$ over $p$. 

The correspondence $c$ has $\phi(U) \subset \TAU_n$ and $\phi(Z) = r$, where $r$ is some $L$-point of $\GR_n$. We need to \glqq move\grqq  $\,\, r$ to the point $q \in \GR_n(k)$ and to \glqq stretch\grqq  $\,\, \phi(U)$, so that $\,\, \phi(U)$ would be embedded into the fiber of $\TAU_n$ over $q$.

\begin{step}
First we  \glqq move\grqq  $\,\, r$ to some point $\hat r \in \GR_n(L)$ such that $\pi_n(\hat r) = p$. Denote $s = \pi_n(r) \in \Gr_n(L)$. The group $\SL_N(L)$ acts transitively on $\Gr(n, N)_L$, so
after taking colimit the group $\SL(L)$ acts transitively on $\Gr_{n, L}$. Thus we can choose a matrix $A \in \SL(L)$ such that $A \cdot s = p$ in $\Gr_{n, L}$. 

The action of $\SL(L)$ on $\Gr_{n, L}$ lifts to an action on $\Tau_{n, L}$ and hence on $\GR_{n, L}$. Since $\Tau_{n, L} \to \Gr_{n, L}$ is a colimit of $\SL(L)$-equivariant vector bundles, the action of $\SL(L)$ extends to $\TAU_{n, L}$. Hence the matrix $A$ gives an automorphism ${\TAU_{n, L} \xrightarrow{A} \TAU_{n, L}}$ where ${A \cdot r} = \hat r \in \GR_{n, L}$ and $\pi_{n, L} (\hat r) = p$. Note that $A$ induces an automorphism of the zero section of  $\TAU_{n, L}$. By Lemma~\ref{lm:sl} there is an equivalence of correspondences $c \sim c_1 = (U,  \phi_1, g)$, where
$$\phi_1(Z) = A \cdot \phi (Z) = \hat r \in \GR_n(L)$$ for $Z = \supp(c) = \supp(c_1)$. 
\end{step}

\begin{step}
Now we  \glqq stretch\grqq  $\,\, \phi_1(U)$, so that it would be embedded in the fiber of $\TAU_n$ over $\hat r$. 
By definition, $\phi$ has image in $\TAU(n, N)$ for some $N \geqslant n$.
The point $p$ has an affine open neighborhood $W \simeq \Aa^m \subset \Gr(n, N)$ for $m = n(N-n)$, over which $\Tau(n, N)$ is canonically trivialized, hence so is $\GR(n, N)$ \cite[Corollary~8.15]{GWAlgGeomI}. 

We have: 
\begin{gather*}
\Tau(n, N) \times_{\Gr(n, N)} W \simeq  \Aa^n \times  \Aa^m; \\
\GR(n, N) \times_{\Gr(n, N)} W \simeq  (\Aa^1-0) \times  \Aa^m; \\
\TAU(n, N) \times_{\GR(n, N)} ((\Aa^1-0) \times  \Aa^m) \simeq \Aa^n \times (\Aa^1-0) \times \Aa^m = V. 
\end{gather*}

We replace $U$ with its open subscheme $U_1 = \phi_1^{-1}(V) \subset U$ in the correspondence $c_1$. 
By Lemma~\ref{lm:projection} we can assume that there is a morphism $h \colon U_1 \to \mathrm{Spec} \, \kappa(Z) \to \spl.$

Let $p$ have coordinates $(0, \dots, 0) \in \Aa^m$. Denote:
 $$\phi_1 = (\rho, \psi, \chi) \colon U_1 \to \Aa^n \times (\Aa^1-0) \times \Aa^m .$$
 Consider the homotopy $d = (U_1 \times \Aa^1, \Phi, g \circ \pr_{U_1}) \in Fr_n^{\SL}(\Aa^1, Y)$, defined by
$$\Phi \colon U_1 \times \Aa^1 \xrightarrow{((\rho,\psi) \circ \pr_{U_1}) ,  (\xi_i)_{i=1}^m}  \Aa^n \times (\Aa^1-0) \times \Aa^m,$$
where $\xi_i(u, t) = (1-t) \cdot (\chi_i(u))$.
Since $\supp(d) = Z \times \Aa^1$, the correspondence $d$ realizes a homotopy between $i_0^*(d) = (U_1, \phi_1, g)$ and $i_1^*(d) = (U_1, (\rho, \psi, p), g) = c_2$, where $p$ denotes the constant map.

Recall that $\phi_1(Z) = \hat r$ where $\hat r$ corresponds to $(l, p) \in (\Aa^1_L-0) \times \Aa^m$ for some $l \in L^{\times}$. Consider the map:
$$ \Psi \colon U_1 \times \Aa^1 \xrightarrow{(1-t) \cdot (\psi, h)(u)+ t \cdot l} \Aa^1_L \xrightarrow{pr} \Aa^1.$$

Denote $U_2 = \Psi^{-1}(\Aa^1-0) \subset U_1 \times \Aa^1$, it is an \'etale neighborhood of $Z \times \Aa^1$. Consider the homotopy: $$d' = (U_2, (\rho \circ \pr', \Psi, p), g \circ \pr') \in Fr_n^{\SL}(\Aa^1, Y),$$ where $\pr'$ denotes the projection $U_2 \hookrightarrow U_1 \times \Aa^1 \to U_1$. We have $\supp(d') = Z \times \Aa^1$, $i_0^*(d') = c_2$, $i_1^*(d') = (U_3,(\rho, \hat r), g) = c_3$. Altogether, we get that $c_1 \sim c_3 = (U_3, \phi_3, g)$ where $\phi_3(U_3) \subset \Aa^n \times \hat r$, which is the fiber of $\TAU_n$ over the point $\hat r \in (\Aa^1-0) \times \Aa^m \subset \GR_n$.
\end{step}

\begin{step}
Finally, we \glqq move\grqq \, the fiber of $\TAU_n$ over $\hat r$ to the fiber over $q = (1, p) \in \Gm \times \Aa^m \subset \GR_n$. Both $\hat r$ and $q$ are in the fiber of $\pi_n$ over $p$.
Consider the embedding $p = \Gr(n, n) \subset \Gr(n, n+1) \simeq \Pp^n$, and note that $$\GR(n, n+1) \simeq \Oo_{\Pp^n}(-1) - z(\Pp^n) \simeq {\Aa^{n+1} - 0}.$$ 

The smooth scheme $\Aa^{n+1}-0$ is $\Aa^1$-chain connected for $n>0$ (see~\cite[Definition~2.2.2]{AMAlgH-cobordism}). 
That means, there is a finite sequence of $\Aa^1_L$-paths $\gamma_0, \dots, \gamma_{\ell}$ in $\GR(n, n+1)$ such that $$\gamma_0(0) = \hat r; \quad \gamma_{\ell}(1) = q; \quad \gamma_i(0) = \gamma_{i-1}(1) \text{ for } 1 \leqslant i \leqslant {\ell}.$$ Each $\Aa^1_L$-path $\gamma_i$ will provide a homotopy that  \glqq moves\grqq \, the fiber of $\TAU_n$ over $\gamma_i(0)$ to the fiber over $\gamma_i(1)$.

Let us fix $0 \leqslant i \leqslant {\ell}$ and consider $\gamma_i \colon \Aa^1_L \to \GR(n, n+1)$.
Every vector bundle has a trivialization over an affine space, so $$\TAU(n, n+1) \times_{\GR(n, n+1)} \Aa^1_L \simeq \Aa^n \times \Aa^1_L.$$ This way we get a homotopy $\Gamma_i \colon \Aa^n \times \Aa^1_L \to  \TAU(n, n+1)$ where $\Gamma_i(\Aa^n, t)$ is the fiber of $\TAU(n, n+1)$ over $\gamma_i(t) \in \GR(n, n+1)(L)$.

Denote $\phi^0 = \pr_{\Aa^n} \circ \phi_3$, $\bar U = U_3$, $c^0 = c_3$. Define $\bar h \colon U_3 \hookrightarrow U_2 \xrightarrow{\pr'} U_1 \xrightarrow{h} \spl$.
Consider the correspondence $d_i = (\bar U \times \Aa^1, \Phi_i, g \circ \pr_{\bar U}) \in Fr_n^{\SL}(\Aa^1, Y),$ given by
$$\Phi_i \colon \bar U \times \Aa^1 \xrightarrow{(\phi^i \circ \pr_1,  \id \circ \pr_2, \bar h \circ \pr_1)} \Aa^n \times \Aa^1 \times \spl \xrightarrow{\Gamma_i} \TAU(n, n+1).$$
We have $\supp(d_i) = Z \times \Aa^1$, $i_0^*(d_i) = c^i$, and we define by induction: $$c^{i+1} = i_1^*(d_i) =(\bar U, \phi^{i+1}, g).$$

The last correspondence $c^{\ell+1}$ has the properties we wanted:
$\phi^{\ell+1}$ maps the support of $c^{\ell+1}$ to $q$, and $\phi^{\ell+1}(\bar U)$ is embedded in the fiber of $\TAU_n$ over $q$. Hence $c^{r+1}$ is in the image of the homomorphism $\varepsilon$, and the proposition follows. 
\end{step}
\end{proof}

\subsection{Finite Milnor-Witt correspondences and framed correspondences}
\label{ssec: milnor-witt}

\begin{notation}
In this subsection we assume that $k$ is a perfect field, $\chr k \ne 2$.
$\K_n^M$ and $\K_n^{MW}$ denote the $n$-th Milnor and Milnor-Witt K-theory groups respectively, defined for all fields, $n \in \Z$. $\Kk_n^{MW}$ denotes the unramified Nisnevich sheaf of Milnor-Witt K-theory on $\smk$, as defined in~\cite[Chapter~2]{MorelA1-alg-top}.
$\gw$ is the presheaf of Grothendieck-Witt groups on $\smk$, its associated Nisnevich sheaf is $\Kk_0^{MW}$.

For a smooth $k$-scheme $X$ we denote by $\Omega_X$ the sheaf of differentials of $X$ over $\spk$, and by $\omega_X = \det \Omega_X$ the canonical sheaf. Given a morphism $f \colon X \to Y$, we write $\omega_f$ or $\omega_{X/Y}$ for $\omega_{X/k} \otimes f^* \omega_{Y/k}^{\vee}$. For an equidimensional scheme $X \in \smk$ we denote $d_X = \dim X$. Finally, $X^{(n)}$ denotes the set of points of codimension $n$.
\end{notation}

\sssec{} Recall the definition of \textit{(twisted) Chow-Witt groups with supports} (see~\cite[Definition~3.1]{CFFiniteMWCor}): for $X \in \smk$, $\Ll$ a line bundle over $X$, $Z \subset X$ a closed subscheme, $n \in \N$ one sets \[\Chw^n_Z(X, \Ll) = H^n_Z(X, \Kk_n^{MW}(\Ll))\] (see~\cite[Section~1.2]{CFFiniteMWCor} for the construction of the twisted sheaf of Milnor-Witt K-theory $\Kk_n^{MW}(\Ll)$). 
The Chow-Witt group $\Chw^n_Z(X, \Ll)$ can be computed as the $n$-th  cohomology group of the Rost-Schmid complex $C^*_{\rs}(X, \Kk_n^{MW}(\Ll))$, constructed in~\cite[Chapter~5]{MorelA1-alg-top}, whose terms are given by \[C^d_{\rs, Z}(X, \Kk_n^{MW}(\Ll)) = \bigoplus_{x \in X^{(d)} \cap Z} (i_x)_*\K_{n-d}^{MW}(\kappa(x), \omega_{x/X} \otimes \Ll).\]
As the classical Chow groups, the Chow-Witt groups with supports are contravariant in $X$ (and $\Ll$), have (twisted) pushforwards along proper maps (more generally, along maps which are proper when restricted to the support), and exterior product which induces the  intersection product.

Let $i \colon Z \hookrightarrow X$ be a closed embedding of codimension $c$ of smooth $k$-schemes. Then the comparison of the corresponding Rost-Schmid complexes gives the \textit{purity isomorphism}
\begin{equation} \label{eq:purity}
\Chw^n_Z(X, \Ll) \simeq  \Chw^{n-c}(Z, i^*\Ll \otimes \det N_i),
\end{equation}
where $N_i$ is the normal bundle of the embedding. 

\sssec{} 
Recall the category of \textit{finite Milnor-Witt correspondences} $\Cor_k$ (see~\cite[Section~4.15]{CFFiniteMWCor}), whose objects are smooth $k$-schemes, and morphisms are given by abelian groups \[ \Cor_k (X, Y) = \colim_{T \in \mathcal{A}(X, Y)} \Chw^{d_Y}_T (X \times Y, \omega_{X \times Y / X} ),\] where $ \mathcal{A}(X, Y)$ is the set of closed subsets of $X \times Y$ that are finite and surjective over corresponding irreducible components of $X$, when endowed with the reduced scheme structure.\footnote{In fact, it doesn't matter which scheme structure on closed subsets to consider in this context.} The category $\Cor_k$ is symmetric monoidal~\cite[Lemma~4.21]{CFFiniteMWCor}. We will write $\Cor(X, Y)$ for $\Cor_k(X, Y)$.

\sssec{} \label{sssec:graph functor}
There is a graph functor $\widetilde{\gamma} \colon \smk \to \Cor_k$ (see~\cite[Section~4.3]{CFFiniteMWCor}), which is defined as identity on objects, and sends a morphism $f \colon X \to Y$ to the pushforward of the quadratic form $\langle 1 \rangle \in \Kk_0^{MW}(X)$ under
\[(\id, f)_* \colon \Kk_0^{MW}(X) = \Chw^0(X) \to \Chw^{d_Y}_{\Gamma_f}(X \times Y, \omega_{X \times Y / X}).\]

\sssec{}
The MW-motivic cohomology $H^{p, q}_{MW} (-, \Z)$ was defined in~\cite[Section~6]{CFFiniteMWCor}. The following analogue of the Nesterenko-Suslin-Totaro theorem~\cite[Theorem~5.1]{MVWMotCohom} was proven in~\cite[Theorem~2.9]{CFComparisonMWCohom}.

\begin{theorem}[Calm\`es, Fasel]
Let $k$ be a perfect field, $\chr k \ne 2$, $L / k$ a finitely generated field extension. Then there is a ring isomorphism, natural in $L$:
$$\Phi_L \colon \bigoplus_{n \in \Z} \K_n^{MW}(L) \xrightarrow{\sim} \bigoplus_{n \in \Z} H^{n, n}_{MW} (\spl, \Z).$$ 
\end{theorem}

\begin{rem}
Note that for $n \geqslant 0$ we have
\[
H^{n, n}_{MW} (\spl, \Z) = H_0(\Cor(\Delta^{\bullet}_L, \Gg_m^{\wedge n})) ,
\]
and the multiplication on $H_0(\Cor(\Delta^{\bullet}_L, \Gg_m^{\wedge *}))$ is defined by means of the exterior product of Chow-Witt groups, in the same way as in Subsection~\ref{ssec: neshitov}. 
\end{rem}

\sssec{}
Here we recall the construction of the functor $$\alpha \colon \Fr_*(k) \longrightarrow \Cor_k,$$ given in~\cite[Proposition~2.1.12]{DF_MWMotComplexes}. On objects one has $\alpha(X) = X$. On correspondences of level $0$ one defines $\alpha$ as the extension of the graph functor $\widetilde{\gamma}$, by mapping correspondences with empty support to $0$. For $c = (U, \phi, g) \in \Fr_n(X, Y)$ a framed correspondence of level $n \geqslant 1$ with support $Z$, we describe how to associate to it  $\alpha(c) \in \Cor(X, Y)$ (it is enough to consider equidimensional $Y$). 

\sssec{} Denote $\phi = (\phi_1, \dots, \phi_n)$, where $\phi_i \in \Oo(U)$, and let $|\phi_i|$ be the vanishing locus of $\phi_i$, then $Z = |\phi_1| \cap \dots \cap |\phi_n|$ as a set. Each $\phi_i \in \oplus_{u \in U^{(0)}} \kappa(u)^{\times}$ defines an element of $\oplus_{u \in U^{(0)}}  \K_1^{MW}(\kappa(u))$. For each $i$ the residue map
$$ \partial \colon \bigoplus_{x \in U^{(0)}} \K_1^{MW}(\kappa(u)) \longrightarrow \bigoplus_{x \in U^{(1)}} \K_0^{MW}(\kappa(x), \omega_{x/X})$$
provides an element $\partial(\phi_i)$ supported on $|\phi_i|$, so  defines a cycle $Z(\phi_i) \in H^1_{|\phi_i|}(U, \K_1^{MW})$. Using the intersection product, we get an element
$$Z(\phi) = Z(\phi_1) \cdot \ldots \cdot Z(\phi_n) \in H^n_Z(U, \K_n^{MW}).$$

As part of the data of $c$, there is an étale map $p \colon U \to \Aa^n_X$. It induces an isomorphism $p^* \omega_{\Aa^n_X} \simeq \omega_U$. Denote the projection by $q \colon \Aa^n_X \to X$. 
On $\Aa^n_X=\mathrm{Spec}_X\Oo_X[t_1,\ldots, t_n]$, the sheaf $\omega_{\Aa^n_X}\otimes q^*\omega_X^\vee$ has the canonical generator $dt_1\wedge\ldots\wedge dt_n$, giving the canonical isomorphism $\Oo_{\Aa^n_X}\simeq  \omega_{\Aa^n_X}\otimes q^*\omega_X^\vee.$
We get the canonical isomorphism:
$$ \Oo_U \simeq p^* (\Oo_{\Aa^n_X}) \simeq 
p^* (\omega_{\Aa^n_X} \otimes q^* \omega_X^{\vee}) \simeq
\omega_U \otimes (q p)^* \omega_X^{\vee}.$$
Thus we can consider $Z(\phi)$ as an element of
 $\Chw^n_Z(U, \, \omega_U \otimes (q p)^* \omega_X^{\vee})$.

The map $(q p, g) \colon U \to X \times Y$ sends $Z$ to a closed subscheme $T$, which is finite and surjective over $X$ by~\cite[Lemma~1.4]{MVWMotCohom}. Since $Z$ is finite over $X$, the restriction $ (q p, g) \big|_{Z} $ is a finite morphism.
 We have then the pushforward morphism:
$$(q p, g)_* \colon  \Chw^n_Z(U,   \, \omega_U \otimes (q p)^* \omega_X^{\vee}) \longrightarrow \Chw^{d_Y}_T(X \times Y,  \, \omega_{X \times Y / X}).$$
The image $(q p, g)_*(Z(\phi))$ is the finite MW-correspondence $\alpha(c) \in \Cor(X, Y)$. 

\sssec{} The functor $\alpha$ is naturally extended to linear framed correspondences.
By~\cite[Example~2.1.11]{DF_MWMotComplexes}, for a suspension morphism $\sigma_Y$ one has $\alpha(\sigma_Y) = \id_Y \in \Cor(Y, Y).$ Altogether, for any $X, Y \in \smk$ we obtain a homomorphism of abelian groups
$$\alpha \colon \ZF(X, Y) \longrightarrow \Cor(X, Y),$$
inducing a homomorphism of simplicial abelian groups
$$\alpha_l \colon \ZF(\del, \Gg_m^{\wedge l}) \to \Cor(\del, \Gg_m^{\wedge l}). $$

For each $l \geqslant 0$ the homomorphism $\alpha_l$ factors through the zeroth homology:
\begin{equation}
\label{eq:alpha_*}
\alpha_* \colon H_0(\ZF(\del, \Gg_m^{\wedge l})) \to 
H_0(\Cor(\del, \Gg_m^{\wedge l})) = H^{l, l}_{MW} (\spk, \Z).
\end{equation}

\begin{lm} \label{lm:alpha_* ring}
The map~\eqref{eq:alpha_*} induces a ring homomorphism:
$$\alpha_* \colon H_0(\ZF(\del, \Gg_m^{\wedge *})) \longrightarrow 
H_0(\Cor(\del, \Gg_m^{\wedge *})).$$
\end{lm}
\begin{proof}
We have to check that for correspondences $c = (U, \phi, g) \in \Fr_n(X, Y)$ and $ d = (V, \psi, h) \in \Fr_m(X_1, Y_1)$ of levels $n, m \geqslant 1$ with non-empty supports $Z$ and $Z'$  holds the following: $$\alpha(c \times d) = \alpha(c) \otimes \alpha(d).$$

First we show that the construction of $Z(\phi)$ respects the product. Since the construction is multiplicative, we can assume that $n = m = 1$. 
The correspondence $c \times d$ has the framing 
$ \chi = (\phi \circ \pr_U, \psi \circ \pr_V) \colon U \times V\longrightarrow \Aa^2$, and is supported on $Z \times Z'$.
Then in $\Chw^2_{Z \times Z'} (U \times V) $ we have:
\begin{multline*}
Z(\chi) = Z(\phi \circ \pr_U) \cdot Z(\psi \circ \pr_V) = [\partial [\phi \circ \pr_U]] \cdot [\partial [\psi \circ \pr_V]] = \\
[\pr_U^* \partial [\phi]] \cdot [\pr_V^* \partial [\psi]] = [\partial [\phi] \times 1_V] \cdot [1_U \times \partial [\psi]] = [\partial [\phi] \times \partial [\psi]] = Z(\phi) \times Z(\psi).
\end{multline*}

The proper pushforward of Chow-Witt groups commutes with exterior product, hence the claim follows.
\end{proof}

\subsection{Injectivity of the unit map $\varepsilon_*$}
\label{ssec: injectivity}

\sssec{} 
In this subsection we assume that $\chr k = 0$.

To construct a left inverse map for $\varepsilon_*$, we will consider the following diagram:
\begin{equation}
\label{diag:injectivity}
\xymatrix{
&H_0(\ZF(\del, \Gg_m^{\wedge *}))  \ar[r]^-{\varepsilon_*} \ar[d]^-{\alpha_*} & H_0(\ZF^{\SL}(\del, \Gg_m^{\wedge *})) \ar@{-->}[ld]^-{\alpha_*^{\SL}} \\
\bigoplus_{l \geqslant 0}  \K_l^{MW}(k)  \ar[ru]^-{\Psi}_-{\sim} \ar[r]^-{\Phi}_-{\sim}
& \bigoplus_{l \geqslant 0} H^{l, l}_{MW} (\spk, \Z) 
}
\end{equation}
Here the isomorphisms $\Psi$ and $\Phi$ are the ones constructed in~\cite[Section~8.3]{NeshitovFramedCorrMW} and~\cite[Theorem~1.8]{CFComparisonMWCohom} respectively. 
We recall how they are constructed on generators $\langle a \rangle \in \gw(k)$ and $[a] \in \K_1^{MW}(k)$ for $a \in k^\times$. 

Denote $\Aa^1 = \spk[x]$ and $\Gg_m = \spk[x, x^{-1}] $.
Then the image $\Psi (\langle a \rangle)$ is the class of the correspondence
 $ (\Aa^1, a x, \pr_k) \in \Fr_1(\spk, \spk)$ in $H_0(\ZF(\del,  \spk))$, and $\Psi ([a])$ is the class of the correspondence $ (\Gm, x-a, \id) \in \Fr_1(\spk, \Gm)$ in $H_0(\ZF(\del,  \Gg_m^{\wedge 1}))$. Meanwhile 
$\Phi(\langle a \rangle) = \langle a \rangle \in \gw(k) = H^{0, 0}_{MW} (\spk, \Z)$ and $\Phi([a])$ is the class of 
$\widetilde{\gamma}(\spk \xrightarrow{a} \Gm) \in \Cor(\spk, \Gm)$ in 
$ H^{1, 1}_{MW}(\spk, \Z)$.

\begin{lm} \label{lm: diag commutes}
With the notations of the diagram~\eqref{diag:injectivity} one has $\Psi \circ \Phi^{-1} \circ  \alpha_* = \id.$
\end{lm}
\begin{proof}
Equivalently, we need to show that $ \Phi =  \alpha_* \circ \Psi$. Since all these maps are ring homomorphisms (see Lemma~\ref{lm:alpha_* ring}), we only need to check that the equation holds for the  generators of $\K_{\geqslant 0}^{MW}(k)$ as a $\Z$-algebra. That is, we need to check it for $\langle a \rangle \in \gw(k)$ and $[a] \in \K_1^{MW}(k)$, where $a \in k^{\times}$ (see~\cite[Section~8.3]{NeshitovFramedCorrMW}). 

1) For $\langle a \rangle \in \gw(k)$ we have to compute  
$ (\alpha_* \circ \Psi) \langle a \rangle = [\alpha(\Aa^1, ax, \pr_k)]$.
Under the residue map
$$ \partial \colon \K_1^{MW}(k(x)) \to \bigoplus_{t \in \Aa^{1(1)}} \K_0^{MW}(k(t), (\m_t / \m_t^2)^{\vee}) $$
we have the following image (see~\cite[Remark~3.21]{MorelA1-alg-top}):
 $$ \partial [ax] = 1 \otimes \overline{ax}^{\vee} = \langle a \rangle \otimes \overline{x}^{\vee} \in \K_0^{MW}(k, (\m_0 / \m_0^2)^{\vee}).$$ 
 After choosing the canonical orientation of $\Aa^1$,
$\langle a \rangle \otimes \overline{x}^{\vee}$ corresponds to the class of $\langle a \rangle \in \Chw^1_0(\Aa^1,  \, \omega_{\Aa^1})$. The pushfoward of $\langle a \rangle$ under
$$(\pr_k)_* \colon \Chw^1_0(\Aa^1,  \, \omega_{\Aa^1}) \to \Chw^0(\spk) = \gw(k)$$
is the class of $\langle a \rangle$, hence 
$\alpha(\Aa^1, ax, \pr_k) = \langle a \rangle \in \gw(k)$, coinciding with $\Phi(\langle a \rangle)$.

2) For $[a] \in \K_1^{MW}(k)$ we have to compute 
$ (\alpha_* \circ \Psi)[a] =[\alpha(\Gm, x-a, \id)]$.
The residue map
$$ \partial \colon \K_1^{MW}(k(x)) \to \bigoplus_{t \in \Gg_m^{(1)}} \K_0^{MW}(k(t), (\m_t / \m_t^2)^{\vee}) $$
gives $$ \partial [x-a] = {1 \otimes \overline{x-a}^{\vee} } \in \K_0^{MW}(k, (\m_a / \m_a^2)^{\vee}),$$ where $a$ is considered as a $k$-point of $\Gm$. 
By construction of the functor $\alpha$, one applies then the isomorphism $\Chw_a^1(\Gm) \simeq \Chw^1_a(\Gm, \, \omega_{\Gm})$, induced by the trivialization $\omega_{\Gm} \simeq \langle dx \rangle$. This way, the class of $ \partial [x-a]$ is given by
$$1 \otimes d(x-a)^{\vee} \otimes dx= \langle 1 \rangle \in \Chw^1_a(\Gm, \, \omega_{\Gm}),$$
since $dx = d(x-a)$.
The pushforward of $\langle 1 \rangle \in \Chw^1_a(\Gm, \, \omega_{\Gm})$ under $(\pr_k, \, \id)_*$ is the same, hence
$$\alpha(\Gm, x-a, \id) = \langle 1 \rangle \in \Chw^1_a(\Gm, \, \omega_{\Gm}) \subset \Cor(\spk, \Gm).$$

On the other hand, $\Phi([a])$ is the class of 
$\widetilde{\gamma}(\spk \xrightarrow{a} \Gm) \in \Cor(\spk, \Gm)$ in 
$ H^{1, 1}_{MW}(\spk, \Z)$. By construction of $\widetilde{\gamma}$, 
$$\widetilde{\gamma}(\spk \xrightarrow{a} \Gm) = \langle 1 \rangle 
\in \Chw^1_a(\Gm,  \,  \omega_{\Gm}).$$
\end{proof}

\sssec{}
Let $\xi \colon E \to X$ be a vector bundle of rank $r$ over a smooth $k$-scheme $X$. The purity isomorphism~\eqref{eq:purity} for the zero section $X \hookrightarrow E$ and the twist by the line bundle $\xi^* \det E^{\vee}$ on $E$ gives the canonical isomorphism:
\[\Chw^0(X) \simeq \Chw^r_X(E, \, \xi^* \det E^{\vee}).\]
The \textit{oriented Thom class} of $\xi \colon E \to X$ is defined as the class $t_{\xi} \in \Chw^r_X(E, \, \xi^* \det E^{\vee})$ that corresponds under the purity isomorphism to the class $ \langle 1 \rangle \in \Chw^0(X)$.

Finally, we have all the tools to prove Theorem~\ref{thm:main}.

\begin{proof}[Proof of Theorem~\ref{thm:main}]

Consider the diagram~\eqref{diag:injectivity}: we know that $\varepsilon_*$ is surjective (Proposition~\ref{prop:epi}) and that the left  triangle commutes (Lemma~\ref{lm: diag commutes}). Hence, it suffices to construct a homomorphism $\alpha_*^{\SL}$ such that the right triangle would commute. 

To do so, we use the alternative construction for the functor $\alpha \colon \Fr_*(k) \to  \Cor_k$ from~\cite[Section~4.3]{deloop2}. Take a correspondence $c = (U, \phi, g) \in \Fr_n(X, Y)$. By~\cite[Lemma~4.3.26]{deloop2}, one can assume that the framing  $\phi$ is a flat map, after refining the \'etale neighborhood $U$ if necessary. In that case, by~\cite[Lemma~4.3.24]{deloop2} one has an equality of cohomology classes  \[Z(\phi) = \phi^* (t_n),\] where $t_n \in \Chw_0^n(\Aa^n)$ is the oriented Thom class of the trivial vector bundle $ \Aa^n \to \spk$. 

Using this description, we construct the functor  
$$\alpha^{\SL} \colon \Fr_*^{\SL}(k) \to \Cor_k$$ as follows. It is identity on objects, and for correspondences of level $0$ we set $\alpha^{\SL} = \alpha$. 
Let $c = (U, \phi, g) \in Fr_n^{\SL}(X, Y)$ have the framing represented by a morphism
$\phi \colon U \to \TAU(n, N)$ for some $N = mp$, and a non-empty support $Z$. 
Denote by $\xi_N \colon \TAU(n, N) \to \GR(n, N)$ the projection 
and recall that there is a trivialization of $\det \TAU(n, N)$, defined in~\eqref{eq:sl-orientation}. 
This trivialization induces a trivialization of the line bundle $\xi_N^* \det \TAU(n, N)^{\vee} \to \TAU(n, N)$. 
 Hence the oriented Thom class of $\xi_N$ is an element of the Chow-Witt group with trivial twist: 
 $t_{\xi_N} \in \Chw^n_{\GR(n, N)}(\TAU(n, N)).$
We define $$Z(\phi) = \phi^*(t_{\xi_N}) \in \Chw^n_Z(U).$$
The cohomology class $Z(\phi)$ does not depend on the choice of $N$, because a  composition with the canonical embedding $i_{N,M} \colon \TAU(n, N) \hookrightarrow \TAU(n, N+M)$ induces an equality
$$(i_{N,M})^*(t_{\xi_{N+M}}) = t_{\xi_N}$$ by~\cite[Proposition~3.7(1)]{LevineEnumGeom}. Applying $\phi^*$ gives us
$$(i_{N,M} \circ \phi)^*(t_{\xi_{N+M}}) = \phi^* ((i_{N,M})^*(t_{\xi_{N+M}})) =\phi^*(t_{\xi_N}).$$

Finally, we set $$\alpha^{\SL}(c) = (q p, g)_*(Z(\phi)) \in \Cor(X, Y),$$ where $p \colon U \to \Aa^n_X$ is the \'etale neighborhood of $Z$ and $q \colon \Aa^n_X \to X$ is the projection. 

By construction, we get an equality of functors $\alpha = \alpha^{\SL} \circ \E \colon \Fr_*(k) \to \Cor_k$, where the functor $\E$ was defined in Section~\ref{sssec: functor E}. The map $\alpha^{\SL}$ factors through stabilization with respect to suspension, and we obtain the induced map \[\alpha^{\SL}_* \colon H_0(\ZF^{\SL}(\del, \Gg_m^{\wedge *})) \longrightarrow H_0(\Cor(\del, \Gg_m^{\wedge *}))\] such that $\alpha_* = \alpha^{\SL}_* \circ \varepsilon_*$. The claim follows.

\end{proof}

\bibliographystyle{alphamod}

\let\mathbb=\mathbf

{\small
\bibliography{ref}
}

\end{document}